\documentclass[10pt]{article}

\usepackage{lmodern}
\usepackage[T1]{fontenc}
\usepackage{amsmath}
\usepackage{amsthm}
\usepackage{amssymb}
\usepackage{mathabx} % For \bigominus
\usepackage{url}
\usepackage{latexsym}
\usepackage{titlefoot}
\usepackage[small]{titlesec}
\usepackage{units} % for \nicefrac
\usepackage[small,it]{caption}
\usepackage[square,comma,numbers,sort&compress]{natbib}

\usepackage{tikz}
\usetikzlibrary{arrows, automata, decorations.pathreplacing, fit, matrix, patterns, positioning}
\usepackage{tikz-qtree}
\usepackage{xifthen} % For if-then control (used in a Tikz macro).

\usepackage{footmisc} % Don't use symbol footnotes

\usepackage{color}
\definecolor{lightgray}{rgb}{0.8, 0.8, 0.8}
\definecolor{darkgray}{rgb}{0.7, 0.7, 0.7}

\usepackage[bookmarks]{hyperref}
\hypersetup{
	colorlinks=true,
	linkcolor=black,
	anchorcolor=black,
	citecolor=black,
	urlcolor=black,
	pdfpagemode=UseThumbs,
	pdftitle={Growth rates of permutation classes: from countable to uncountable},
	pdfsubject={Combinatorics},
	pdfauthor={Vincent Vatter},
}

\newcounter{todocounter}

\newcommand{\minisec}[1]{\noindent{\sc #1.}}

\theoremstyle{plain}
\newtheorem{theorem}{Theorem}[section]
\newtheorem{proposition}[theorem]{Proposition}

\newtheorem{corollary}[theorem]{Corollary}
\newtheorem{conjecture}[theorem]{Conjecture}

\newtheorem{problem}[theorem]{Problem}
\newtheorem{observation}[theorem]{Observation}

\newtheorem*{esthm*}{The Erd\H{o}s--Szekeres Theorem~\cite{erdos:a-combinatorial:}}
\newtheorem*{thm-xi-main*}{Theorem~\ref{thm-xi-main}}
\newtheorem*{egf*}{Exponential Growth Formula}
\newtheorem*{higmans-lemma*}{Higman's Lemma~\cite{higman:ordering-by-div:}}

\makeatletter
\newfont{\footsc}{cmcsc10 at 8truept}
\newfont{\footbf}{cmbx10 at 8truept}
\newfont{\footrm}{cmr10 at 10truept}
\renewenvironment{abstract}{
	\begin{list}{}%
	{\setlength{\rightmargin}{1in}%
	\setlength{\leftmargin}{1in}}%
	\item[]\ignorespaces\begin{small}}%
	{\end{small}\unskip\end{list}%
}

\newcommand{\Av}{\operatorname{Av}}
\newcommand{\Grid}{\operatorname{Grid}}
\newcommand{\Geom}{\operatorname{Geom}}
\newcommand{\Sub}{\operatorname{Sub}}
\newcommand{\PropSub}{\Sub^{<}}
\newcommand{\A}{\mathcal{A}}
\newcommand{\C}{\mathcal{C}}
\newcommand{\D}{\mathcal{D}}
\newcommand{\E}{\mathcal{E}}
\newcommand{\G}{\mathcal{G}}

\newcommand{\K}{\mathcal{K}}
\newcommand{\LL}{\mathcal{L}}
\newcommand{\M}{\mathcal{M}}

\renewcommand{\O}{\mathcal{O}}

\newcommand{\gr}{\mathrm{gr}}
\newcommand{\lgr}{\underline{\gr}}
\newcommand{\ugr}{\overline{\gr}}
\newcommand{\zpm}{0/\mathord{\pm} 1}
\newcommand{\st}{\::\:}

\newcommand{\bij}{\varphi}

\newcommand{\gridded}{\sharp}

\newcommand{\emptyword}{\varepsilon}

\renewcommand{\l}{\mathsf{l}}
\renewcommand{\r}{\mathsf{r}}

\newcommand{\emptygray}{\text{\textcolor{darkgray}{$\emptyset$}}}
\newcommand{\verteq}{\rotatebox{90}{$=\,$}}
\newcommand{\fnmatrix}[2]{\text{\begin{footnotesize}$\left(\begin{array}{#1}#2\end{array}\right)$\end{footnotesize}}}

\newcommand\mybullet{\raisebox{-5pt}{\normalsize \ensuremath{\bullet}}}
\newcommand\mycirc{\raisebox{-5pt}{\normalsize \ensuremath{\circ}}}

\makeatletter
\def\absdot{\@ifnextchar[{\@absdotlabel}{\@absdotnolabel}}
	\def\@absdotlabel[#1]#2{%
		\node at #2 {\normalsize \mybullet};
		\node at #2 [below=2pt] {\ensuremath{#1}};
	}
	\def\@absdotnolabel#1{%
		\node at #1 {\normalsize \mybullet};
	}
\def\absdothollow{\@ifnextchar[{\@absdothollowlabel}{\@absdothollownolabel}}
	\def\@absdothollowlabel[#1]#2{%
		\node at #2 {\normalsize \textcolor{white}{\mybullet}};
		\node at #2 {\normalsize \mycirc};
		\node at #2 [below=2pt] {\ensuremath{#1}};
	}
	\def\@absdothollownolabel#1{%
		\node at #1 {\normalsize \textcolor{white}{\mybullet}};
		\node at #1 {\normalsize \mycirc};
	}
\makeatother

\newcommand{\plotperm}[1]{
	\foreach \j [count=\i] in {#1} {
		\absdot{(\i,\j)};
	};
}

\newcommand{\plotpartialperm}[1]{
	\foreach \i/\j in {#1} {
		\absdot{(\i,\j)};
	};
}

\newcommand{\plotpermbox}[4]{
	\draw [darkgray, very thick, rounded corners=0.01, line cap=round]
		({#1-0.5}, {#2-0.5}) rectangle ({#3+0.5}, {#4+0.5});
}

\newcommand{\plotpermgraph}[1]{
	\foreach \j [count=\i] in {#1} {
		\foreach \b [count=\a] in {#1} {
			\ifthenelse{\a<\i \AND \b>\j}{\draw (\a,\b)--(\i,\j);}{}
		};
	};
	\plotperm{#1};
}

\newcommand{\plotpermdyckpath}[1]{
	\draw [ultra thick, rounded corners=0.01, line cap=round] (0.5,0.5)
	\foreach \step in {#1} {
		\ifnum\step=1
			-- ++(0,1)
		\else
			-- ++(1,0)
		\fi
	};
}

\newcommand{\matrixpermwithzeros}[2]{
	\foreach \y [count=\x] in {#2} {
		\foreach \j in {1, 2, ..., #1} {
			\ifthenelse{\j=\y}{
				\node at (\x,\j) {\tiny $\mathbf{1}$};
			}{
				\node at (\x,\j) {\textcolor{darkgray}{\tiny $0$}};
			}
		}
	}
}

\newcommand{\plotdyckpath}[1]{
	\draw[ultra thick, line cap=round] (0.5,0)
	\foreach \step in {#1} {
		\ifnum\step=1
			-- ++(1,1)
		\else
			-- ++(1,-1)
		\fi
	};
}

\newcommand{\arcskinnyplain}[2]{
	\draw (#1,0) arc (180:0:{(#2-#1)/2});
}

\newcommand{\matchsmall}[1]{
	\begin{tikzpicture}[scale=.1, anchor=base]
		\def\h{0};
		\def\maxh{0};
		\foreach \i/\j in {#1} {
			\pgfmathparse{\j-\i};
			\let\h\pgfmathresult;
			\pgfmathifthenelse{\h>\maxh}{\h}{\maxh};
			\global\let\maxh\pgfmathresult;
		};
		\pgftransformyscale{{4.5/\maxh}};
		\foreach \i/\j in {#1} {
			\arcskinnyplain{\i}{\j};
		};
	\end{tikzpicture}
}

\newcommand{\matchpermsmall}[1]{
	\begin{tikzpicture}[scale=.1, anchor=base]
		\foreach \j [count=\n] in {#1} {};
		\def\h{0};
		\def\maxh{0};
		\foreach \j [count=\i] in {#1} {
			\pgfmathparse{2*\n+1-\j-\i};
			\let\h\pgfmathresult;
			\pgfmathifthenelse{\h>\maxh}{\h}{\maxh};
			\global\let\maxh\pgfmathresult;
		};
		\pgftransformyscale{{4.5/\maxh}};
		\foreach \j [count=\i] in {#1} {
			\arcskinnyplain{\i}{{2*\n+1-\j}};
		};
	\end{tikzpicture}
}

\newcommand{\plotpinsequence}[1]{
	\absdot{(0,0)}{};
	\edef\n{0}
	\edef\s{0}
	\edef\e{0}
	\edef\w{0}
	\edef\x{0}
	\edef\y{0}
	\foreach \pin [remember=\pin as \oldpin (initially 1), count=\i] in {#1} {
		\ifthenelse{\pin=1 \OR \pin=2}{%up or down
			\ifthenelse{\oldpin=3}{% previous=right
				\xdef\x{\number\numexpr\e-1}
			}{
				\xdef\x{\number\numexpr\w+1}
			}
			\ifnum\i=1 %expand eastern box by 1 if 1st pin
				\pgfmathparse{\e+1}
 				\xdef\e{\pgfmathresult}
			\fi	
		}{ %left or right
			\ifthenelse{\oldpin=1}{% previous=up
				\xdef\y{\number\numexpr\n-1}
			}{
				\xdef\y{\number\numexpr\s+1}
			}
			\ifnum\i=1 %expand southern boundary by 1 if 1st pin
				\pgfmathparse{\s-1}
 				\xdef\s{\pgfmathresult}
			\fi	
		}
		\ifnum\pin=1 %up
			\pgfmathparse{\n+2}
 			\xdef\n{\pgfmathresult}		
			\absdot{(\x,\n)}{};
			\ifnum\i>1
				\draw (\x,\n) -- (\x,\y-0.5);
			\else
			\fi
		\fi
		\ifnum\pin=2 % down		
			\pgfmathparse{\s-2}
 			\xdef\s{\pgfmathresult}
			\absdot{(\x,\s)}{};
			\ifnum\i>1
				\draw (\x,\s) -- (\x,\y+0.5);
			\else
			\fi
		\fi
		\ifnum\pin=3 %right
			\pgfmathparse{\e+2}
 			\xdef\e{\pgfmathresult}
			\absdot{(\e,\y)}{};
			\ifnum\i>1
				\draw (\e,\y) -- (\x-0.5,\y);
			\else
			\fi
		\fi
		\ifnum\pin=4 %left
			\pgfmathparse{\w-2}
 			\xdef\w{\pgfmathresult}
			\absdot{(\w,\y)}{};
			\ifnum\i>1
				\draw (\w,\y) -- (\x+0.5,\y);
			\else
			\fi
		\fi		
	};
}

\newcommand{\gridsmallhoriz}[1]{
	\begin{tikzpicture}[scale=1, anchor=base]
    \def\gridheight{1};
    \foreach \dir [count=\gridwidth] in {#1} {
    };
	  \pgftransformxscale{{0.225/\gridwidth}};
		\pgftransformyscale{{0.225/\gridheight}};
    \foreach \dir [count=\i] in {#1} {
      \ifthenelse{\dir>0}{
		    \draw [semithick, line cap=round] ({\i-1}, 0)--(\i, 1);
      }{
        \draw [semithick, line cap=round] ({\i-1}, 1)--(\i, 0);
      };
    };
  \end{tikzpicture}
}

\newcommand{\gridhoriz}[1]{
	\begin{tikzpicture}[scale=1, anchor=base]
	  \pgftransformxscale{0.225};
		\pgftransformyscale{0.225};
    \foreach \dir [count=\i] in {#1} {
      \ifthenelse{\dir>0}{
		    \draw [semithick, line cap=round] ({\i-1}, 0)--(\i, 1);
      }{
        \draw [semithick, line cap=round] ({\i-1}, 1)--(\i, 0);
      };
    };
  \end{tikzpicture}
}

\newcommand{\gridsmallvert}[1]{
	\begin{tikzpicture}[scale=1, anchor=base]
    \def\gridwidth{1};
    \foreach \dir [count=\gridheight] in {#1} {
    };
	  \pgftransformxscale{{0.225/\gridwidth}};
		\pgftransformyscale{{0.225/\gridheight}};
    \foreach \dir [count=\i] in {#1} {
      \ifthenelse{\dir>0}{
		    \draw [semithick, line cap=round] (0, {\i-1})--(1, \i);
      }{
        \draw [semithick, line cap=round] (0, \i)--(1, {\i-1});
      };
    };
  \end{tikzpicture}
}

\newcommand{\gridverysmallvert}[1]{
	\begin{tikzpicture}[scale=1, anchor=base]
    \def\gridwidth{1};
    \foreach \dir [count=\gridheight] in {#1} {
    };
	  \pgftransformxscale{{0.175/\gridwidth}};
		\pgftransformyscale{{0.175/\gridheight}};
    \foreach \dir [count=\i] in {#1} {
      \ifthenelse{\dir>0}{
		    \draw [semithick, line cap=round] (0, {\i-1})--(1, \i);
      }{
        \draw [semithick, line cap=round] (0, \i)--(1, {\i-1});
      };
    };
  \end{tikzpicture}
}

\datefoot{\today}
\amssubj{05A05, 05A16 (primary), 05A15 (secondary).}
\newpagestyle{main}[\small]{
	\headrule
	\sethead[\usepage][][]
	{\sc Growth Rates of Permutation Classes}{}{\usepage}
}

\title{\sc Growth Rates of Permutation Classes:\\ From Countable to Uncountable}
\author{%
Vincent Vatter%
\footnote{The author was partially supported by the National Science Foundation under Grant Number DMS-1301692.}\\[-0.25ex]
\small Department of Mathematics\\[-0.5ex]
\small University of Florida\\[-0.5ex]
\small Gainesville, Florida USA\\[-1.5ex]
}

\titleformat{\section}{\large\sc}{\thesection.}{1em}{}
\date{}

\begin{document}
\maketitle

\addtocounter{footnote}{1}

\pagestyle{main}

\begin{abstract}	
We establish that there is an algebraic number $\xi\approx 2.30522$ such that while there are uncountably many growth rates of permutation classes arbitrarily close to $\xi$, there are only countably many less than $\xi$. Central to the proof are various structural notions regarding generalized grid classes and a new property of permutation classes called concentration. The classification of growth rates up to $\xi$ is completed in a subsequent paper.
\end{abstract}

\section{Introduction}

Scheinerman and Zito~\cite{scheinerman:on-the-size-of-:} established a seminal result on the enumeration of hereditary properties of graphs (sets of graphs closed under isomorphism and induced subgraphs) in 1994. They showed that the enumeration of every hereditary property of labelled graphs follows one of five patterns: eventually constant, eventually polynomial, exponential, factorial, or superfactorial. In the decades since, this result has been refined and extended to numerous other contexts. In particular, Balogh, Bollob\'as, and Morris~\cite{balogh:hereditary-prop:ordgraphs,balogh:hereditary-prop:part,balogh:hereditary-prop:posets,balogh:hereditary-prop:} have established analogues of Scheinerman and Zito's results for ordered graphs and hypergraphs, oriented graphs, posets, set partitions, and tournaments. We refer to Bollob\'as' survey~\cite{bollobas:hereditary-and-:BCC:} for a detailed history of these results. This work is primarily concerned with permutations, although advances in this context also inform the closely-related context of ordered graphs.

A \emph{permutation class}\footnote{``Permutation class'' is the analogue of ``hereditary property'' in this context. Where the terms used in the permutation community differ from those used in the graph theory community we have tended to adopt the permutation-specific term.} is a set of finite permutations closed downward under the \emph{containment order}, in which we say that $\sigma$ is contained in $\pi$ if $\pi$ contains a (not necessarily consecutive) subsequence which is \emph{order isomorphic} to $\sigma$ (i.e., has the same relative comparisons), and in this case we call the subsequence a \emph{copy} of $\sigma$. We say that the permutation $\pi$ \emph{avoids} $\sigma$ it if does not contain it. For example, the permutation $41523$ contains $231$ because of its subsequence $452$ but avoids $321$ because it does not contain three entries in decreasing order. Given a permutation class $\C$ we let $\C_n$ denote the set of permutations of length $n$ in $\C$, and call the function $n\mapsto |\C_n|$ the \emph{speed} of $\C$.

The resolution of the Stanley--Wilf Conjecture by Marcus and Tardos~\cite{marcus:excluded-permut:} (using a result of Klazar~\cite{klazar:the-furedi-hajn:}) shows that speeds of proper permutation classes grow at most exponentially. Our interest here is in determining the possible exponential bases of these speeds. To this end, we define the \emph{upper} and \emph{lower growth rates} of $\C$ by
\[
	\ugr(\C) = \limsup_{n\rightarrow\infty}\sqrt[n]{|\C_n|}
	\quad\mbox{and}\quad
	\lgr(\C) = \liminf_{n\rightarrow\infty}\sqrt[n]{|\C_n|}.
\]
It is conjectured these limits are always equal, i.e., that every permutation class has a \emph{growth rate} (also called a \emph{Stanley--Wilf limit}). When we are dealing with a class for which $\ugr(\C)=\lgr(\C)$ we denote this quantity by $\gr(\C)$.

The most common way to define a permutation class is by avoidance, that is, as the set
\[
	\Av(B)=\{\pi\st\pi\mbox{ avoids all $\beta\in B$}\}
\]
of permutations that avoid every member of set $B$. We may here assume that $B$ is an \emph{antichain}, meaning that no element of $B$ contains any others; in this case $B$ is unique and is called the \emph{basis} of the class.

There has been considerable interest in growth rates of \emph{principal} classes, which are those with singleton bases. In particular, the class $\Av(4231)$ has received significant attention, though only rough bounds on its growth rate have been established. The best of these bounds to-date---due to Bevan, Brignall, Elvey Price, and Pantone~\cite{bevan:a-structural-ch:} (which improved on an earlier lower bound by Bevan~\cite{bevan:a-large-set-of-:} and an earlier upper bound by B\'ona~\cite{bona:a-new-record-fo:})---are still far from the estimates of Conway and Guttmann~\cite{conway:on-the-growth-r:} and Conway, Guttmann, and Zinn-Justin~\cite{conway:1324-avoiding-p:}, based on differential approximants, and Madras and Liu~\cite{madras:random-pattern-:}, based on Monte Carlo approximation. In a more general setting, Fox~\cite{Fox:Stanley--Wilf-l:} has shown that the growth rates of principal classes are typically much larger than had been thought. To be clear: we are not interested in growth rates of specific classes here, but rather with the set of growth rates of all permutation classes, principal \emph{and otherwise}.

There are several other ways to specify permutations classes. Indeed, the field dates its origins to Knuth~\cite{knuth:the-art-of-comp:1}, who specified permutation classes using simple sorting machines. For the purposes of this work, two alternative specifications are important. First, given any set $X$ of permutations, one way to obtain a permutation class is to take the \emph{downward closure} of $X$,
\[
	\Sub(X)
	=
	\{\pi\st\pi\le\tau\mbox{ for some $\tau\in X$}\}.
\]
In fact, given any injection $f\st A\rightarrow B$ between linearly ordered sets $A$ and $B$, we can define $\Sub(f\st A\rightarrow B)$ as the set of all finite permutations which are order isomorphic to the restriction of $f$ to a finite set.

The final type of specification we use is geometric. We say that a \emph{figure} is a subset $\Phi$ of the plane. We further say that the permutation $\pi$ of length $n$ can be \emph{drawn} on $\Phi$ if there are $n$ points of $\Phi$, no two on a common horizontal or vertical line, so that if we label the points $1$ to $n$ reading from bottom to top and then record their labels reading left to right we obtain $\pi$. (In this sense, the permutation $\pi$ is essentially acting as the modulus in a moduli space.) The class of all permutations which can be drawn on the figure $\Phi$ is denoted $\Sub(\Phi)$. This viewpoint is used most significantly in Section~\ref{sec-concentration}.

The first general result on permutation classes was established over 80 years ago:

\begin{esthm*}
Every permutation of length at least $(k-1)^2+1$ contains either $12\cdots k$ or $k\cdots 21$.
\end{esthm*}

The first general result on growth rates of permutation classes was proved much more recently by Kaiser and Klazar~\cite{kaiser:on-growth-rates:} in 2003. For a positive integer $k$, the \emph{$k$-generalized Fibonacci numbers} are defined by $F_{n,k}=0$ for negative $n$, $F_{0,k}=1$, and
\[
	F_{n,k}=F_{n-1,k}+F_{n-2,k}+\cdots+F_{n-k,k}
\]
for $n\ge 1$. (One obtains the ordinary Fibonacci numbers by setting $k=2$.) Kaiser and Klazar's result establishes that speeds of permutation classes of growth rate less than $2$ are severely constrained.

\begin{theorem}[Kaiser and Klazar~\cite{kaiser:on-growth-rates:}]
\label{thm-kk-fib}
For every permutation class $\C$, one of the following holds.
\begin{enumerate}
\item[(1)] The speed of $\C$ is bounded.
\item[(2)] The speed of $\C$ is eventually polynomial (specifically, there is an integer $N$ and a polynomial $p(n)$ such that $|\C_n|=p(n)$ for all $n\ge N$).
\item[(3)] There are integers $k\ge 2$ and $\ell\ge 1$ such that $F_{n,k}\le |\C_n|\le n^\ell F_{n,k}$ for every $n$.
\item[(4)] For every $n$, $|\C_n|\ge 2^{n-1}$.
\end{enumerate}
\end{theorem}

In particular, this result implies what is known as the \emph{Fibonacci Dichotomy}: if $|\C_n|$ is less than the $n$th Fibonacci number for any value of $n$ then the speed of $\C$ is eventually polynomial. This implication of Kaiser and Klazar's Theorem~\ref{thm-kk-fib} was later reproved by Huczynska and Vatter~\cite{huczynska:grid-classes-an:} using precursors of the techniques employed here.

Theorem~\ref{thm-kk-fib} has since been generalized to the context of ordered graphs. An \emph{ordered graph} is a graph together with a linear order on its vertices. By convention, we always take the vertices of an $n$-vertex ordered graph to be $[n]=\{1,2,\dots,n\}$ with the standard integer order. The ordered graph $H$ on the vertices $[k]$ is contained as an \emph{ordered induced subgraph} in the ordered graph $G$ on $[n]$ if there is an increasing injection $f\st[k]\rightarrow[n]$ such that $f(i)\sim f(j)$ if and only if $i\sim j$ (we denote the edge relation by $\sim$). To match our permutation notation, we refer to downsets of ordered graphs under this order as \emph{ordered graph classes}.

Given a permutation $\pi$ on $[n]$, its \emph{inversion} (or \emph{permutation}) \emph{graph} $G_\pi$ is the graph on the vertices $[n]$ in which $i\sim j$ if and only if the entries $\pi(i)$ and $\pi(j)$ form an inversion in $\pi$, i.e., if and only if $i<j$ and $\pi(i)>\pi(j)$. Note that this correspondence is not injective; in particular, the inversion graphs $G_\pi$ and $G_{\pi^{-1}}$ are isomorphic for every permutation $\pi$. The \emph{ordered inversion graph} of $\pi$ consists of $G_\pi$ equipped with the order $1<2<\cdots<n$ on its vertices. The correspondence between permutations and ordered inversion graphs is an order-preserving injection, so every permutation class is isomorphic (as a partially ordered set) to its corresponding ordered graph class. However, there are ordered graphs which do not arise as ordered inversion graphs, so there are ordered graph classes which do not correspond to permutation classes. Nevertheless, Balogh, Bollob\'as, and Morris showed in \cite{balogh:hereditary-prop:ordgraphs} that ordered graph classes of growth rate at most $2$ also fall into the categorization provided by Theorem~\ref{thm-kk-fib}.

Kaiser and Klazar's Theorem~\ref{thm-kk-fib} shows that the only possible growth rates (these are indeed proper growth rates) of permutation classes (and thus also of ordered graph classes) less than $2$ are $0$, $1$, and the unique real roots of the polynomials $x^k-x^{k-1}-\cdots-x-1$ for $k\ge 2$ (equivalently, these are the greatest real roots of the polynomials $x^{k+1}-2x^k+1$). These growth rates accumulate at $2$, making it the least accumulation point of growth rates.

Kaiser and Klazar's result was extended in Vatter~\cite{vatter:small-permutati:}, which characterized all possible growth rates (again, these are necessarily proper) up to
\[
	\kappa=\text{the unique positive root of $x^3-2x^2-1$}\approx 2.20557.
\]
Classes with these growth rates (i.e., less than $\kappa$) are called \emph{small permutation classes} and their growth rates have more complex structure. In particular, the set of growth rates up to $\kappa$ has infinitely many accumulation points, which themselves accumulate at $\kappa$, making it the least \emph{second-order accumulation point} of growth rates.

\begin{theorem}[Vatter~\cite{vatter:small-permutati:}]
\label{thm-spc-gr}
The growth rates of small permutation classes consist precisely of $0$, $1$, $2$, and roots of the following families of polynomials (for all nonnegative $k$ and $\ell$):
\begin{enumerate}
\item[(1)] $x^{k+1}-2x^k+1$ (the sub-$2$ growth rates of Theorem~\ref{thm-kk-fib}),
\item[(2)] $(x^3-2x^2-1)x^{k+\ell}+x^\ell+1$,
\item[(3)] $(x^3-2x^2-1)x^k+1$ (accumulation points of growth rates of type (2) which themselves accumulate at $\kappa$), and
\item[(4)] $x^4-x^3-x^2-2x-3$, $x^5-x^4-x^3-2x^2-3x-1$, $x^3-x^2-x-3$, and $x^4-x^3-x^2-3x-1$ .
\end{enumerate}
\end{theorem}

In addition to $\kappa$ being the least second-order accumulation point, several other notable phase transitions occur there. It is the point at which infinite antichains first occur and consequently, the point where there are uncountably many permutation classes. As shown by Albert, Ru\v{s}kuc, and Vatter~\cite{albert:inflations-of-g:}, $\kappa$ is also the point at which permutation classes may first begin to have nonrational generating functions.

In the ordered graph context, no characterization of growth rates greater than $2$ is known. However, it is known that there is an ordered graph class of growth rate approximately $2.03166$, while there is no permutation class of that growth rate\footnote{An ordered graph class with this growth rate is constructed in the conclusion of \cite{vatter:small-permutati:}. Conjecture 8.7 of Balogh, Bollob\'as, Morris~\cite{balogh:hereditary-prop:ordgraphs} asserts that it is the least growth rate of an ordered graph class above $2$.}, so the ordered graph context truly becomes more general for growth rates above $2$.

On the large side of the growth rate spectrum, Albert and Linton~\cite{albert:growing-at-a-pe:} disproved Balogh, Bollob\'as, and Morris' conjecture from \cite{balogh:hereditary-prop:ordgraphs} that the set of growth rates has no accumulation points from above\footnote{Klazar~\cite{klazar:overview-of-som:} suggested a possible way to salvage their conjecture, but this is disproved in our subsequent paper, Pantone and Vatter~\cite{pantone:growth-rates-of:}.}, and showed that this set is uncountable. In Vatter~\cite{vatter:permutation-cla:lambda:}, a refinement of Albert and Linton's approach was used to show that every real number greater than $2.48188$ is the growth rate of a permutation class. Bevan later related these constructions to the famous base $\beta$ representations of numbers studied by R\'enyi~\cite{renyi:representations:} and established the following.

\begin{theorem}[Bevan~\cite{bevan:intervals-of-pe:}]
\label{thm-lambda-B}
The set of growth rates of permutation classes contains every real number at least the unique positive root of $x^8-2x^7-x^5-x^4-2x^3-2x^2-x-1$, $\lambda_B\approx 2.35698$.
\end{theorem}

In addition, Bevan exhibited an infinite sequence of nontrivial intervals of growth rates of permutation classes whose infimum is approximately $2.35526$.

\begin{figure}
\begin{footnotesize}
\begin{center}
	\input{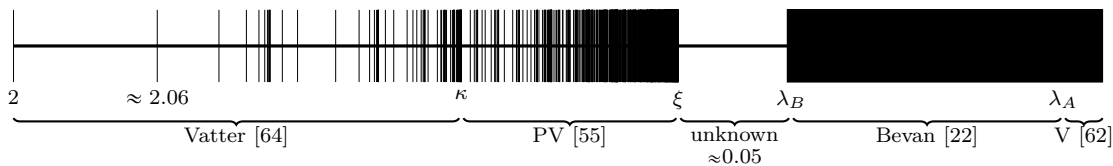}
\end{center}
\end{footnotesize}
\caption[The set of all growth rates of permutation classes between $2$ and $2.5$, as it stands after the characterization in Pantone and Vatter.]{The set of all growth rates of permutation classes between $2$ and $2.5$, as it stands after the characterization in Pantone and Vatter~\cite{pantone:growth-rates-of:}.}
\label{fig-set-of-growth-rates}
\end{figure}

In this paper we explore the gap between $\kappa$ and $\lambda_B$. As we show, a fundamental qualitative change occurs at the growth rate
\[
	\xi=\text{the unique positive root of $x^5-2x^4-x^2-x-1$}\approx 2.30522.
\]
In \cite{vatter:permutation-cla:lambda:}\footnote{The approximation presented in \cite{vatter:permutation-cla:lambda:} contains a typo, which is why this correct approximation differs slightly from what appears there.} it was established that there are uncountably many growth rates in every open neighborhood of $\xi$. The main result of this paper shows that this behavior occurs precisely at $\xi$. Indeed, it says more than that, as it shows that every upper growth rate less than $\xi$ is both a proper growth rate (i.e., the limit exists) and that it is the growth rate of a sum closed class, a concept defined in the following section.

\begin{thm-xi-main*}
There are only countably many growth rates of permutation classes below $\xi$ but uncountably many growth rates in every open neighborhood of it. Moreover, every growth rate of permutation classes below $\xi$ is achieved by a sum closed permutation class.
\end{thm-xi-main*}

The proof of Theorem~\ref{thm-xi-main} is completed in Section~\ref{sec-phase-transition}, after we have built up a collection of tools to analyze the classes of interest. In \cite{pantone:growth-rates-of:} we establish the complete list of growth rates between $\kappa$ and $\xi$, and in doing so establish that each of these growth rates is achieved by a finitely based class, whereas there are growth rates arbitrarily close to $\xi$ which cannot be achieved by finitely based classes (because there are uncountably many such growth rates). This characterization and Bevan's Theorem~\ref{thm-lambda-B} leave us tantalizingly close to the complete determination of the set of possible growth rates of permutation classes, as shown in Figure~\ref{fig-set-of-growth-rates}.

\section{Background}
\label{sec-background}

In this section we provide a brief review of some basic concepts of permutation classes that arise in our arguments. For a thorough introduction to permutation classes, including all of the notions reviewed in this section, we refer the reader to the author's survey~\cite{vatter:permutation-cla:} in the \emph{CRC Handbook of Enumerative Combinatorics}.

Much of the research on permutation classes has focused on their exact and asymptotic enumeration. Given a class $\C$ of permutations, we denote by $\C_n$ the set of permutations in $\C$ of length $n$. The \emph{generating function} of $\C$ is then
\[
	\sum_{\pi\in \C} x^{|\pi|}=\sum_{n\ge 0}|\C_n|x^n,
\]
where $|\pi|$ denotes the length of $\pi$. The Exponential Growth Formula, presented in specialized form below, connects exact and asymptotic enumeration.

\begin{egf*}[see Flajolet and Sedgewick~{\cite[Section IV.3.2]{flajolet:analytic-combin:}}]
The upper growth rate of a permutation class is equal to the reciprocal of the least positive singularity of its generating function.
\end{egf*}

We say that the permutation $\tau$ is an \emph{$p$-point extension} of $\pi$ if $\tau$ contains $\pi$ and has length at most $p$ greater than the length of $\pi$, i.e., if $\tau$ can be formed by inserting $p$ or fewer new entries into $\pi$. Given a class $\C$, the $p$-point extension of $\C$, denoted by $\C^{+p}$,  is the class consisting of all $p$-point extensions of members of $\C$. Clearly $|\C_n^{+1}|\le (n+1)^2|\C_n|$, from which we obtain the following.

\begin{proposition}
\label{prop-extension-gr}
For every permutation class $\C$ and every positive integer $p$ we have $\ugr(\C^{+p})=\ugr(\C)$ and $\lgr(\C^{+p})=\lgr(\C)$.
\end{proposition}

We now review the \emph{substitution decomposition}%
\footnote{This decomposition---which dates back to a 1953 talk of Fra{\"{\i}}ss{\'e}~\cite{fraisse:on-a-decomposit:} and found a notable application in Gallai's 1967 paper~\cite{gallai:transitiv-orien:,gallai:a-translation-o:}---has been studied for a many different types of combinatorial object, but we restrict our discussion to the permutation case here.}%
. An \emph{interval} in the permutation $\pi$ is a set of contiguous indices $I=[a,b]$ such that the set of values $\pi(I)=\{\pi(i) : i\in I\}$ is also contiguous. Every permutation $\pi$ of length $n$ has intervals of length $0$, $1$, and $n$ which are called \emph{trivial}, and a permutation is \emph{simple} if it has no other intervals.

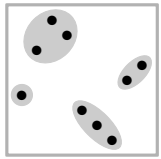
\begin{figure}
\begin{center}
	\begin{tikzpicture}[scale=0.2]
		% The 1:
		\draw [lightgray, fill] (1,4) circle (20pt);
		% The 132:
		\draw[lightgray, fill, rotate around={-45:(2.9,7.9)}] (2.9,7.9) ellipse (45pt and 55pt);
		% The bottom 321:
		\draw[lightgray, fill, rotate around={45:(6,2)}] (6,2) ellipse (25pt and 60pt);
		% Right most 12:
		\draw[lightgray, fill, rotate around={-45:(8.5,5.5)}] (8.5,5.5) ellipse (20pt and 40pt);
		% The permutation:
		\plotpermbox{0.5}{0.5}{9.5}{9.5};
		\plotperm{4,7,9,8,3,2,1,5,6};
	\end{tikzpicture}
\end{center}
\caption{The plot of $479832156$, an inflation of $2413$.}
\label{fig-479832156}
\end{figure}

Given a permutation $\sigma$ of length $m$ and nonempty permutations $\alpha_1,\dots,\alpha_m$, the \emph{inflation} of $\sigma$ by $\alpha_1,\dots,\alpha_m$ --- denoted $\sigma[\alpha_1,\dots,\alpha_m]$ --- is the permutation obtained by replacing each entry $\sigma(i)$ by an interval that is order isomorphic to $\alpha_i$. For example, $2413[1,132,321,12]=4\ 798\ 321\ 56$ (see Figure~\ref{fig-479832156}). The following uniqueness result was one of the first in Albert and Atkinson's seminal investigation of the substitution decomposition of permutations in \cite{albert:simple-permutat:}.

\begin{proposition}[Albert and Atkinson~\cite{albert:simple-permutat:}]\label{simple-decomp-1}
Every permutation except $1$ is the inflation of a unique simple permutation of length at least $2$. Furthermore, if $\pi$ can be expressed as $\sigma[\alpha_1,\dots,\alpha_m]$ where $\sigma$ is a nonmonotone simple permutation then each interval $\alpha_i$ is unique.
\end{proposition}

Given two classes $\C$ and $\D$, the \emph{inflation} of $\C$ by $\D$ is defined as
\[
\C[\D]=\{\sigma[\alpha_1,\dots,\alpha_m]\st\mbox{$\sigma\in\C_m$ and $\alpha_1,\dots,\alpha_m\in\D$}\}.
\]

The class $\C$ of permutations is \emph{substitution closed} if $\sigma[\alpha_1,\dots,\alpha_m]\in\C$ for all $\sigma\in\C_m$ and $\alpha_1,\dots,\alpha_m\in\C$. The \emph{substitution closure} of $\C$, denoted $\langle\C\rangle$, is defined as the smallest substitution closed class containing $C$. Thus $\C$ and $\langle\C\rangle$ contain the same simple permutations, and one can also define $\langle\C\rangle$ as the largest permutation class containing the same simple permutations as $\C$. Finally, one can define the substitution closure of $\C$ constructively as
\[
	\langle\C\rangle
	=
	\C\cup\C[\C]\cup\C[\C[\C]]\cup\cdots.
\]

Two inflations in particular play prominent roles in our work: the \emph{sum} $\pi\oplus\sigma=12[\pi,\sigma]$ and the \emph{skew sum} $\pi\ominus\sigma=21[\pi,\sigma]$. A class $\C$ is said to be \emph{sum closed} if $\pi\oplus\sigma\in\C$ for all $\pi,\sigma\in\C$ and \emph{skew closed} if $\pi\ominus\sigma\in\C$ for all $\pi,\sigma\in\C$. Furthermore, a permutation is \emph{sum indecomposable} if it cannot be expressed as the sum of two shorter permutations and \emph{skew sum indecomposable} (or simply \emph{skew indecomposable}) if it cannot be expressed as the skew sum of two shorter permutations. Note that the inversion graph $G_{\pi\oplus\sigma}$ consists of the disjoint union of the inversion graphs $G_\pi$ and $G_\sigma$. Indeed, it follows readily that $\pi$ is sum indecomposable if and only if its inversion graph is connected.

Given a permutation class $\C$, its \emph{sum closure}, denoted by $\bigoplus\C$, is the smallest sum closed class containing $\C$. Equivalently,
\[
	\bigoplus\C
	=
	\{\pi_1\oplus\pi_2\oplus\cdots\oplus\pi_k : \mbox{$\pi_i\in\C$ for all $i$}\}.
\]
We analogously define the \emph{skew closure of $\C$} to be the smallest skew closed class containing $\C$ and denote this class by $\bigominus\C$. Counting sum (equivalently, skew) closed classes is routine, given the sequence of sum indecomposables:

\begin{proposition}
\label{prop-enum-oplus-closure}
The generating function for a sum closed class is $1/(1-g)$ where $g$ is the generating function for nonempty sum indecomposable permutations in the class.
\end{proposition}

If $\C$ is a sum closed class, then the injection $(\pi,\sigma)\mapsto\pi\oplus\sigma$ shows that the sequence $\{|\C_n|\}$ is \emph{supermultiplicative}, meaning that $|\C_m||\C_n|\le |\C_{m+n}|$ for all $m$ and $n$. As Arratia~\cite{arratia:on-the-stanley-:} was the first to observe, Fekete's Lemma implies that if the sequence $\{a_n\}$ is supermultiplicative then $\lim\sqrt[n]{a_n}$ exists (and equals $\sup\sqrt[n]{a_n}$). Thus we obtain the following.

\begin{proposition}[Arratia~\cite{arratia:on-the-stanley-:}]
\label{prop-arratia-gr}
Every sum or skew closed permutation class has a (possibly infinite) growth rate.
\end{proposition}%

In Section~\ref{sec-phase-transition} we show that all growth rates below $\xi$ are  growth rates of sum closed classes. Thus our characterization of these growth rates in \cite{pantone:growth-rates-of:} reduces to a characterization of the possible sequences of sum indecomposable classes, though that effort is far from trivial. It is not a coincidence that $\xi$ is the growth rate of a sum closed class whose sum indecomposable members have the enumeration $1,1,2,3,4^\infty$ (where $4^\infty$ denotes an infinite sequence of $4$s), though it is perhaps a coincidence that the same growth rate is achieved by a sum closed class whose sum indecomposable members have the enumeration $1,1,2,4,3,3,2,1$.

It is frequently helpful to note that permutation containment is preserved by the symmetries of the square, i.e., the dihedral group with $8$ elements. This group is generated by the three symmetries reverse, complement, and inverse; given a permutation $\pi$ of length $n$ the \emph{reverse} of $\pi$ is the permutation $\pi^{\textrm r}$ defined by $\pi^{\textrm r}(i) = \pi(n+1-i)$, the \emph{complement} of $\pi$ is the permutation $\pi^{\textrm c}$ defined by $\pi^{\textrm c}(i) = n + 1 - \pi(i)$, and the (group-theoretic) \emph{inverse} of $\pi$ is the permutation $\pi^{-1}$ defined by $\pi^{-1}(\pi(i)) = \pi(\pi^{-1}(i)) = i$.

\begin{figure}
\begin{center}
	\begin{tikzpicture}[scale=0.2]
		\plotperm{5, 10, 3, 8, 2, 6, 1, 9, 4, 7};
		\plotpermbox{1}{1}{10}{5};
		\plotpermbox{1}{6}{10}{10};
	\end{tikzpicture}
\quad\quad
	\begin{tikzpicture}[scale=0.2]
		\plotperm{1, 10, 2, 9, 3, 8, 4, 7, 5, 6};
		\plotpermbox{1}{1}{10}{5};
		\plotpermbox{1}{6}{10}{10};
	\end{tikzpicture}
\quad\quad
	\begin{tikzpicture}[scale=0.2]
		\plotperm{1, 3, 5, 7, 9, 2, 4, 6, 8, 10};
		\plotpermbox{1}{1}{5}{10};
		\plotpermbox{6}{1}{10}{10};
	\end{tikzpicture}
\quad\quad
	\begin{tikzpicture}[scale=0.2]
		\plotperm{2,4,1,6,3,8,5,10,7,9};
		\plotpermbox{1}{1}{10}{10};
	\end{tikzpicture}
\quad\quad
	\begin{tikzpicture}[scale=0.2]
		\plotperm{8, 10, 6, 9, 4, 7, 2, 5, 1, 3};
		\plotpermbox{1}{1}{10}{10};
	\end{tikzpicture}
\end{center}
\caption{From left to right, an arbitrary vertical alternation, a vertical wedge alternation, a horizontal parallel alternation, an increasing oscillation, and a decreasing oscillation.}
\label{fig-alternation-oscillation}
\end{figure}
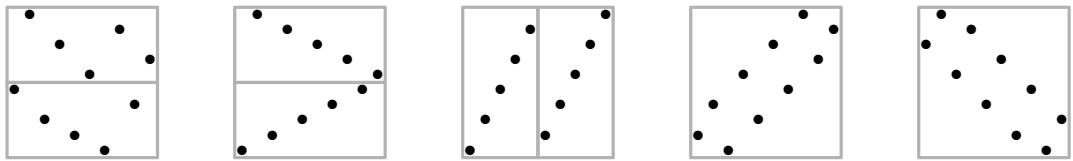

A \emph{vertical alternation} is a permutation in which every entry of odd index lies above every entry of even index, or the complement of such a permutation. For example, the permutations shown in the left and center of Figure~\ref{fig-alternation-oscillation} are vertical alternations. A \emph{horizontal alternation} is the inverse of a vertical alternation, i.e., a permutation in which every entry of odd value lies to the left of every entry of even value, or the reverse of such a permutation. The rightmost permutation drawn in Figure~\ref{fig-alternation-oscillation} is a horizontal alternation. We use the term \emph{alternation} refers to a permutation that is either a vertical or horizontal alternation.

Equivalently, an alternation is a permutation whose plot can be divided into two parts, by a single horizontal or vertical line, so that for every pair of entries from the same part there is an entry from the other part which \emph{separates} them, i.e., there is an entry from the other part which lies either horizontally or vertically between them. A \emph{parallel alternation} is an alternation in which these two sets of entries form monotone subsequences, either both increasing or both decreasing, while a \emph{wedge alternation} is one in which the two sets of entries form monotone subsequences pointing in opposite directions. Our next result follows routinely from the Erd\H{o}s--Szekeres Theorem.

\begin{proposition}
\label{prop-alt-par-wedge}
If a permutation class contains arbitrarily long alternations then it contains arbitrarily long parallel or wedge alternations.
\end{proposition}
%\begin{proof}
%Suppose that $\C$ contains an alternation of length $2n\ge 2k^4$. By symmetry, we may assume that this alternation is a vertical alternation. The Erd\H{o}s--Szekeres Theorem then implies that the sequence $\pi(1),\pi(3),\dots,\pi(2n-1)$ contains a monotone subsequence of length at least $k^2$, say $\pi(i_1),\pi(i_2),\dots,\pi(i_{k^2})$. Applying the Erd\H{o}s--Szekeres Theorem to the subsequence $\pi(i_1+1),\pi(i_2+1),\dots,\pi(i_{k^2}+1)$ completes the proof.
%\end{proof}

We need another special family of permutations, those whose inversion graphs are paths. If the inversion graph $G_\pi$ is a path, we call $\pi$ an \emph{increasing oscillation}. This term dates back to Murphy's thesis~\cite{murphy:restricted-perm:}, though note that under our definition, the permutations $1$, $21$, $231$, and $312$ are increasing oscillations while in other works they are not. By direct construction, the increasing oscillations can be seen to be precisely the sum indecomposable permutations which are order isomorphic to subsequences of the \emph{increasing oscillating sequence},
\[
	2,4,1,6,3,8,5,\dots,2k,2k-3,\dots.
\]
Symmetrically, there are two \emph{decreasing oscillations} of each length; these are the reverses of increasing oscillations. Examples are shown on the right of Figure~\ref{fig-alternation-oscillation}.

We let $\O_I$ denote the downward closure of the set of increasing oscillations, $\O_D$ the downward closure of the set of decreasing oscillations, and $\O=\O_I\cup\O_D$ the downward closure of all oscillations. Note that all oscillations of length at least $4$ are simple permutations. Indeed, together with alternations they constitute unavoidable subpermutations of simple permutations:

\begin{theorem}[Brignall, Ru\v{s}kuc, and Vatter~{\cite[Theorem 13]{brignall:simple-permutat:decide:}}]
\label{thm-alt-or-osc}
If a permutation class contains arbitrarily long simple permutations then it contains arbitrarily long alternations or oscillations\footnote{The graph theoretical analogue of this result was later established by Chudnovsky, Kim, Oum, and Seymour~\cite{chudnovsky:unavoidable-ind:}.}.
\end{theorem}

Combining Theorem~\ref{thm-alt-or-osc} with Proposition~\ref{prop-alt-par-wedge} we see that if a class contains arbitrarily long simple permutations then it contains arbitrarily long oscillations, parallel alternations, or wedge alternations. The converse does not hold because wedge alternations are not simple. Nevertheless a procedure to decide if a class (specified by a finite basis) contains infinitely many simple permutations was outlined in \cite{brignall:simple-permutat:decide:}, and a more efficient procedure was later described by Bassino, Bouvel, Pierrot, and Rossin~\cite{bassino:an-algorithm-fo:}.

We make use of two more results concerning simple permutations themselves. The first is quite elementary.

\begin{proposition}[Vatter~\cite{vatter:small-configura:}]
\label{prop-simple-four}
Every entry in a simple permutation $\pi$ of length at least $4$ is contained in a copy of $2413$ or $3142$ or participates as the \emph{`$3$'} in a copy of $25314$ or $41352$.
% The \emph there is to prevent the apostrophes from looking weird.
\end{proposition}

The second was first proved in the much more general context of binary relational structures; a more elementary proof specialized to the permutation context is given in Brignall and Vatter~\cite{brignall:a-simple-proof-:}.

\begin{theorem}[Schmerl and Trotter~\cite{schmerl:critically-inde:}]\label{thm-schmerl-trotter}
Every simple permutation of length $n\ge 5$ which is not a parallel alternation contains a simple permutation of every length $5\le m\le n$.
\end{theorem}

Another concept which plays a central role in this work is that of well-quasi-order. Recall that a quasi-order (a reflexive and transitive binary relation) is \emph{well-quasi-ordered} (or forms a \emph{well-quasi-order}) if it has neither infinite antichains (sets of pairwise incomparable members) nor infinite strictly decreasing chains. Note that the containment order on permutations is a partial order and hence a quasi-order. Furthermore, as the containment order does not allow for infinite strictly decreasing chains, in this context well-quasi-order is synonymous with the lack of infinite antichains.

A fundamental result used to establish well-quasi-order is Higman's Lemma~\cite{higman:ordering-by-div:}. Suppose that $(P,\le)$ is any poset. Then the \emph{generalized subword order} on $P^\ast$ is defined by $v\le w$ if there are indices $1\le i_1<i_2<\cdots<i_{|v|}\le|w|$ such that $v(j)\le w(i_j)$ for all $j$. (The regular subword order on words over the alphabet $\Sigma$ is obtained from this general version by taking the poset $(\Sigma,\le)$ to be an antichain.) We then have the following.

\begin{higmans-lemma*}
If $(P,\le)$ is well-quasi-ordered then $P^*$, ordered by the subword order, is also well-quasi-ordered.
\end{higmans-lemma*}

The following is one of the first results on well-quasi-ordering of permutations. We use a generalization of it (Theorem~\ref{thm-subst-geom-inflate-wqo}) in Section~\ref{sec-wqo}.

\begin{theorem}[Albert and Atkinson~\cite{albert:simple-permutat:}]
\label{thm-fin-simples-wqo}
Every permutation class with only finitely many simple permutations is well-quasi-ordered.	
\end{theorem}

We conclude this section with two consequences of well-quasi-order. Both are essentially folklore.

\begin{proposition}
\label{prop-wqo-subclasses-countable}
A well-quasi-ordered permutation class contains only countably many subclasses.
\end{proposition}
\begin{proof}
Let $\C$ be a well-quasi-ordered permutation class. Every subclass of $\C$ can be expressed as $\C\cap\Av(B)$ for an antichain $B\subseteq\C$. Because $\C$ is well-quasi-ordered, every such set $B$ must be finite, so there are only countably many subclasses.	
\end{proof}

\begin{proposition}
\label{prop-wqo-subclasses-dcc}
The subclasses of any well-quasi-ordered permutation class satisfy the \emph{descending chain condition}, i.e., if $\C$ is a well-quasi-ordered class, there does not exist an infinite sequence $\C=\C^{(0)}\supsetneq \C^{(1)}\supsetneq \C^{(2)}\supsetneq\cdots$ of permutation classes.
\end{proposition}
\begin{proof}
Suppose to the contrary that the well-quasi-ordered class $\C$ were to contain an infinite strictly decreasing sequence of subclasses $\C=\C^{(0)}\supsetneq \C^{(1)}\supsetneq \C^{(2)}\supsetneq\cdots$. For each $i\ge 1$, choose $\beta_i\in\C^{(i-1)}\setminus\C^{(i)}$. The set of minimal elements of $\{\beta_1,\beta_2\ldots\}$ is an antichain and therefore finite, so there is an integer $m$ such that $\{\beta_1,\beta_2\ldots,\beta_m\}$ contains these minimal elements. In particular, $\beta_{m+1}\ge\beta_i$ for some $1\le i\le m$. However, we chose $\beta_{m+1}\in\C^{(m)}\setminus\C^{(m+1)}$, and because $\beta_{m+1}$ contains $\beta_i$, it does not lie in $\C^{(i)}$ and thus cannot lie in $\C^{(m)}$, a contradiction.
\end{proof}

\section{Gridding Intermediate Classes}
\label{sec-gridding}

The process of gridding a permutation class consists of partitioning the entries of each of its members into a bounded number of rectangles, each containing a subsequence order isomorphic to a member of a specified class. In this section we show that all classes with growth rates less than $\xi$ can be gridded into relatively well-behaved classes. In Section~\ref{sec-slicing} we further refine these griddings using the results established in the intervening sections.

Before describing this concept further we must establish some notation. First, in order for indices of matrices to align with plots of permutations, we index matrices in Cartesian coordinates. Thus the entry $\M_{3,2}$ of the matrix $\M$ is the entry in the $3$rd column from the left and the $2$nd row from the bottom. Next, given a permutation $\pi$ of length $n$ and sets $I,J\subseteq [n]$, we write $\pi(I\times J)$ for the permutation that is order isomorphic to the subsequence of $\pi$ with indices from $I$ and values in $J$. For example, to compute $391867452([3,7]\times[2,6])$ we consider the subsequence of entries in indices $3$ through $7$, $18674$, which have values between $2$ and $6$; in this case the subsequence is $64$ so $391867452([3,7]\times[2,6])=21$.

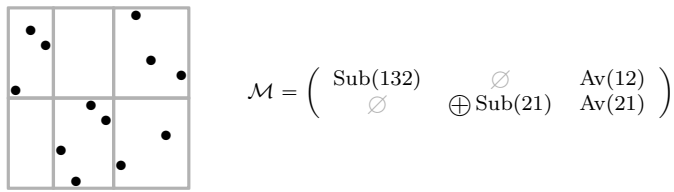
\begin{figure}
\begin{footnotesize}
\begin{center}
	\begin{tikzpicture}[scale=0.2,baseline=(current bounding box.center)]
		\plotperm{7,11,10,3,1,6,5,2,12,9,4,8};
		\plotpermbox{1}{1}{3}{6};
		\plotpermbox{1}{7}{3}{12};
		\plotpermbox{4}{1}{7}{6};
		\plotpermbox{4}{7}{7}{12};
		\plotpermbox{8}{1}{12}{6};
		\plotpermbox{8}{7}{12}{12};
	\end{tikzpicture}
\quad\quad
	$\M=\left(\begin{array}{ccc}
		\Sub(132)&\emptygray&\Av(12)\\
		\emptygray&\bigoplus\Sub(21)&\Av(21)\end{array}
	\right)$
\end{center}
\end{footnotesize}
\caption{The picture on the left shows an $\M$-gridded permutation for the matrix $\M$ on the right.}
\label{fig-M-gridding-example}
\end{figure}

Let $\mathcal{M}$ be a $t\times u$ matrix of permutation classes. An \emph{$\mathcal{M}$-gridding} of the permutation $\pi$ of length $n$ is a choice of \emph{column divisions} $1=c_1\le\cdots\le c_{t+1}=n+1$ and \emph{row divisions} $1=r_1\le\cdots\le r_{u+1}=n+1$ such that for all $i$ and $j$, $\pi([c_i,c_{i+1})\times [r_j,r_{j+1}))$ lies in the class $\mathcal{M}_{i,j}$. Figure~\ref{fig-M-gridding-example} shows an example. A more geometric way to describe this gridding would be to say that the \emph{grid lines} $x=c_i+\nicefrac{1}{2}$ and $y=r_i+\nicefrac{1}{2}$ divide the plot of the permutation into cells. The cells themselves are \emph{axis-parallel rectangles} because each of their sides is parallel to either the $x$- or $y$-axis.

The class of all permutations which possess $\mathcal{M}$-griddings is the (generalized) \emph{grid class} of $\mathcal{M}$, denoted $\Grid(\mathcal{M})$. The class $\C$ is said to be \emph{${\G}$-griddable} if $\C\subseteq\Grid(\mathcal{M})$ for some finite matrix $\mathcal{M}$ whose entries are all equal to $\G$. In this context we call $\G$ the \emph{cell class} of the gridding.

The concept of grid classes was first studied in this generality in order to characterize the growth rates of small classes in \cite{vatter:small-permutati:}, though monotone grid classes---in which the cells are restricted to monotone classes---had been studied previously by Murphy and Vatter~\cite{murphy:profile-classes:}, Huczynska and Vatter~\cite{huczynska:grid-classes-an:}, Vatter and Waton~\cite{vatter:on-partial-well:}, and Brignall~\cite{brignall:grid-classes-an:}. In this context it is convenient to abbreviate $\Av(21)$, the class of increasing permutations, by $1$, $\Av(12)$ by $-1$, and the empty class by $0$. In doing so, we specify monotone grid classes as $\zpm$ matrices. In this manner, the two matrices
\[
	\fnmatrix{ccc}{
		\Av(21)&\emptyset&\Av(12)\\
		\Av(12)&\Av(12)&\Av(21)
	}
	\quad\text{and}\quad
	\fnmatrix{ccc}{
		1&0&-1\\
		-1&-1&1
	}
\]
represent the same monotone grid class. Bevan~\cite{bevan:growth-rates-of:geom,bevan:growth-rates-of:} has since found beautiful connections between growth rates of monotone grid classes and algebraic graph theory (see also Albert and Vatter~\cite{albert:an-elementary-p:} where an elementary proof of Bevan's theorem is given). We use another type of grid classes, geometric grid classes, in the arguments of Section~\ref{sec-wqo}.

Here we present the results on generalized grid classes required later, with sketches of proofs for completeness. We begin with the criterion for a class to be $\G$-griddable.

\begin{theorem}[Vatter~{\cite[Theorem~3.1]{vatter:small-permutati:}}]
\label{thm-gridding-characterization}
The permutation class $\C$ is $\G$-griddable if and only if it does not contain arbitrarily long sums or skew sums of basis elements of $\G$, that is, if and only if there is a constant $m$ so that $\C$ does not contain $\beta_1\oplus\cdots\oplus\beta_m$ or $\beta_1\ominus\cdots\ominus\beta_m$ for any sequence $\beta_1,\dots,\beta_m$ of basis elements of $\G$.
\end{theorem}

To sketch the proof of Theorem~\ref{thm-gridding-characterization}, we begin by defining the set
\[
	\mathfrak{R}_\pi=\{\mbox{axis-parallel rectangles $R$}\st \pi(R)\notin\G\}.
\]
For every rectangle $R\in\mathfrak{R}_\pi$, $\pi(R)$ contains a basis element of $\G$, so in any $\G$-gridding of $\pi$ every rectangle in $\mathfrak{R}_\pi$ must be sliced by a grid line. The converse also holds---if every rectangle in $\mathfrak{R}_\pi$ is sliced by a grid line these grid lines define a $\G$-gridding of $\pi$---so $\C$ is $\G$-griddable if and only if there is a constant $\ell$ such that, for every $\pi\in\C$, the set $\mathfrak{R}_\pi$ can be sliced by $\ell$ horizontal and vertical lines.

We say that two rectangles are \emph{independent} if both their $x$- and $y$-axis projections are disjoint, and a set of rectangles is said to be independent if they are pairwise independent. An \emph{increasing sequence} of rectangles is a sequence $R_1,\dots,R_m$ of independent rectangles such that $R_{i+1}$ lies above and to the right of $R_i$ for all $i$. \emph{Decreasing sequences} of rectangles are defined analogously.

If $\C$ contains arbitrarily long sums or skew sums of basis elements of $\G$ then we see that $\mathfrak{R}$ contains arbitrarily large independent sets, and thus gridding $\C$ by $\G$ would require arbitrarily many grid lines. Now suppose that $\C$ does not contain sum or skew sums consisting of more than $m$ basis elements of $\G$.

Every independent set of rectangles corresponds to a permutation. Thus we see by the Erd\H{o}s--Szekeres Theorem that $\mathfrak{R}_\pi$ cannot contain an independent set of size greater than $(m-1)^2$. The proof of Theorem~\ref{thm-gridding-characterization} is therefore completed by the following corollary of a result of Gy{\'a}rf{\'a}s and Lehel, which was later strengthened by K{\'a}rolyi and Tardos~\cite{karolyi:on-point-covers:} and by Tardos~\cite{tardos:transversals-of:}. For the precise connection between Theorem~\ref{rectangles-lemma} and the results of Gy{\'a}rf{\'a}s and Lehel~\cite{gyarfas:a-helly-type-pr:} we refer to Vatter~\cite{vatter:an-erdos--hajna:}.

\begin{theorem}[Gy{\'a}rf{\'a}s and Lehel~\cite{gyarfas:a-helly-type-pr:}]
\label{rectangles-lemma}
There is a function $f(m)$ such that for any collection $\mathfrak{R}$ of axis-parallel rectangles in the plane that has no independent set of size greater than $m$, there exists a set of $f(m)$ horizontal and vertical lines that slice every rectangle in $\mathfrak{R}$.
\end{theorem}

Next we consider the task of gridding all classes of upper growth rate less than $\gamma$ for an arbitrary real number $\gamma$. Define
\[
	\G_\gamma=\{\pi\st \mbox{either $\gr\left(\bigoplus\Sub(\pi)\right)<\gamma$ or $\gr\left(\bigominus\Sub(\pi)\right)<\gamma$}\}.
\]
We aim to show that $\G_\gamma$ can serve as the cell class to grid all classes of upper growth rate less than $\gamma$. Clearly every member of $\G_\gamma$ is necessary for this goal; if $\pi\in\G_\gamma$ then either $\bigoplus\Sub(\pi)$ or $\bigominus\Sub(\pi)$ has growth rate less than $\gamma$ and in order to grid either of these classes we must have $\pi$ in the cell class. Our next result shows that this is sufficient.

\begin{proposition}
\label{prop-G-gamma-grids}
For every real number $\gamma$, if the permutation class $\C$ satisfies $\ugr(\C)<\gamma$ then $\C$ is $\G_\gamma$-griddable.
\end{proposition}
\begin{proof}
Suppose to the contrary that $\C$ satisfies $\ugr(\C)<\gamma$ but is not $\G_\gamma$-griddable, and let $f$ denote the generating function of $\C$. Now fix an arbitrary integer $m$. We seek to derive a contradiction by showing that $f(\nicefrac{1}{\gamma})>m$, which, because $m$ is arbitrary, will imply that $\ugr(\C)\ge\gamma$. By Theorem~\ref{thm-gridding-characterization} and symmetry, we know that there is a permutation of the form $\beta_1\oplus\cdots\oplus\beta_m$ contained in $\C$ where each $\beta_i$ is a basis element of $\G_\gamma$. Furthermore, because $\beta_i\notin\G_\gamma$, both $\gr(\bigoplus\Sub(\beta_i))$ and $\gr(\bigominus\Sub(\beta_i))$ are at least $\gamma$ (although we only use the first of these inequalities).

Let $g_i$ denote the generating function for the nonempty sum indecomposable permutations of $\Sub(\beta_i)$ for every index $i$, so that the generating function for $\bigoplus\Sub(\beta_i)$ is given by $1/(1-g_i)$. By the Exponential Growth Formula, each generating function $1/(1-g_i)$ has a positive singularity in the interval $(0,\nicefrac{1}{\gamma}]$ and because each $g_i$ is a polynomial, this singularity must be a pole. Moreover, each $g_i$ has positive coefficients, so this implies that the unique positive solution to $g_i(x)=1$ occurs for $x\le\nicefrac{1}{\gamma}$. In particular, $g_i(\nicefrac{1}{\gamma})\ge 1$ for all indices $i$.

Now consider the set of permutations of the form $\alpha_1\oplus\cdots\oplus\alpha_k$ for some $k\le m$, where $\alpha_i$ is a sum indecomposable permutation contained in $\beta_i$ for each index $i$. Clearly this is a subset (but likely not a subclass) of $\C$. Moreover, the generating function for this set of permutations is
\[
	g_1+g_1g_2+\cdots+g_1\cdots g_m,
\]
which is at least $m$ when evaluated at $\nicefrac{1}{\gamma}$, completing the contradiction.
\end{proof}

Two facts about the cell classes $\G_\gamma$ arise immediately from their definition. First, every class of the form $\G_\gamma$ is closed under all eight symmetries of the square. Second, because $\bigoplus\Sub(\pi)=\bigoplus\Sub(\pi\oplus\pi)$ and $\bigominus\Sub(\pi)=\bigominus\Sub(\pi\ominus\pi)$, for every $\pi\in\G_\gamma$ either $\bigoplus\Sub(\pi)$ or $\bigominus\Sub(\pi)$ is contained in $\G_\gamma$. Thus every class of the form $\G_\gamma$ is the union of some number of sum closed classes together with some number of skew closed classes.

%Basis of $\G_2$: 14352, 52134, 24351, 14532, 42351, 34251, 312564, 15423, 231654, 24531, 34125, 53124, 32451, 15243, 51342, 14523, 51243, 32415, 42315, 52314, 465213, 3142, 34215, 231645, 51423, 456213, 321645, 13542, 25341, 42135, 2413, 41325, 321564, 15342, 546132, 456132, 23541, 465123, 15324, 24315, 32541, 546123, 51324, 312654, 52143, 43125. G.F.: (8*x**6 + 6*x**5 - 5*x**4 - 3*x**3 - 2*x**2 + 3*x - 1)/((x - 1)*(2*x - 1)*(x**2 + x - 1))

In this language, one of the first steps in the classification of small permutation classes was verifying that $\G_\kappa$ is contained in the substitution closure of the oscillations (the class $\O$ defined in Section~\ref{sec-background}).

\begin{proposition}[Vatter~{\cite[Proposition~A.3]{vatter:small-permutati:}}]
\label{Gkappa-first}
The cell class $\G_\kappa$ is contained in $\langle\O\rangle$.
\end{proposition}

Because the sequence of sum indecomposable permutations contained in an increasing oscillation of length $i\ge 4$ is $1,1,2^{i-3},1$, we have $\O_I\subseteq\G_\kappa$. By symmetry, we also have $\O_D\subseteq\G_\kappa$. However, the growth rates of sum (resp., skew) closures of increasing (resp., decreasing) oscillations tend to $\kappa$, so the following fact is a consequence of Proposition~\ref{Gkappa-first}.

\begin{proposition}
\label{prop-Ggamma-simples}
The cell class $\G_\gamma$ contains finitely many simple permutations if and only if $\gamma<\kappa$.
\end{proposition}

Needless to say, Proposition~\ref{Gkappa-first} is not the end of the story regarding $\G_\kappa$; later results from \cite{vatter:small-permutati:} add considerably more restrictions. Indeed, while these restrictions proved sufficient to establish the desired goal, we do not have a concrete description of $\G_\kappa$ (and such a concrete description would seem to have questionable value).

As Proposition~\ref{Gkappa-first} provided for $\G_\kappa$, our next result is our first, but far from our last, result about $\G_\xi$. Note that by Bevan's Theorem~\ref{thm-lambda-B} this result actually holds for all cell classes up to the point where all real numbers are growth rates of permutation classes, at or before $\lambda_B\approx 2.35698$.

%Note: $\G_\kappa$ has: 2 basis elements of length 5, 266 basis elements of length 6, 188 of length 7, 336 of length 8, 112 of length 9, 796 of length 10, 396 of length 11, NONE of length 12.

%Note: $\G_\xi$ has: 158 basis elements of length 6, 608 of length 7, 460 of length 8, 1660 of length 9, 2084 of length 10, 1132 of length 11. We think it must be infinitely based.

\begin{proposition}
\label{prop-Gxi-first}
The cell class $\G_{2.36}$ is contained in $\langle\Av(321)\cup\Av(123)\rangle\cup\{25314, 41352\}$.
\end{proposition}
\begin{proof}
We prove two facts. First, we show no proper inflations of $25314$ or $41352$ lie in $\G_{2.36}$. Then we show that the simple permutations of $\G_{2.36}$ all lie in $\Av(321)\cup\Av(123)\cup\{25314, 41352\}$. The first task is straight-forward. If $\G_{2.36}$ were to contain a proper inflation of $25314$ or $41352$ then it would contain a permutation in which one of the entries of these two permutations was inflated by $12$ or $21$. By symmetry, we may assume that $25314$ is the inflated permutation and that the entry is either the greatest entry or the `$3$'. The following chart eliminates all four possibilities.
\[
	\begin{array}{lllll}
	\hline
	\mbox{inflation}&\mbox{sequence of sum}&\mbox{sum closure}&\mbox{sequence of skew}&\mbox{skew closure}
	\\
	\mbox{of $25314$}&\mbox{indecomposables}&\mbox{growth rate}&\mbox{indecomposables}&\mbox{growth rate}
	\\\hline
	256314&1,1,3,6,4,1&\approx 2.44874&1,1,3,4,3,1&\approx 2.36772
	\\
	265314&1,1,3,5,4,1&\approx 2.41421&1,1,3,4,3,1&\approx 2.36772
	\\
	263415&1,1,3,6,3,1&\approx 2.43245&1,1,3,4,3,1&\approx 2.36772
	\\
	264315&1,1,3,4,3,1&\approx 2.36772&1,1,3,6,3,1&\approx 2.43245
	\\\hline
	\end{array}
\]

Next we show that the simple permutations of $\G_{2.36}$ are contained in $\Av(321)\cup\Av(123)\cup\{25314, 41352\}$. For this we use a result of Atkinson, Ru\v{s}kuc, and Smith~\cite{atkinson:substitution-cl:}. They showed that the substitution closure of a principally based class is typically infinitely based. Fortunately, $\langle\Av(321)\rangle$ is an exception\footnote{The graph theoretical analogue of this result, establishing the minimal forbidden induced subgraphs for the substitution closure of triangle-free graphs, was given by Olariu~\cite{olariu:on-the-closure-:}.}; by bounding the length of potential basis elements of this class and then conducting an exhaustive computer search, they established that
\[
	\langle\Av(321)\rangle
	=
	\Av(25314, 35142, 41352, 42513, 362514, 531642).
\]
The following chart presents the relevant data for the basis elements of $\langle\Av(321)\rangle$.
\[
	\begin{array}{lll}
	\hline
	\langle\Av(321)\rangle\mbox{ basis}&\mbox{sequence of sum}&\mbox{sum closure}
	\\
	\mbox{element}&\mbox{indecomposables}&\mbox{growth rate}
	\\\hline
	25314&1,1,3,3,1&\approx 2.29237
	\\
	41352=25314^{\textrm{r}}&1,1,3,3,1&\approx 2.29237
	\\
	35142&1,1,3,5,1&\approx 2.66917>2.36
	\\
	42513=35142^{\textrm{rc}}&1,1,3,5,1&\approx 2.66917>2.36
	\\
	362514&1,1,3,8,6,1&\approx 2.52326>2.36
	\\
	531642=362514^{-1}&1,1,3,8,6,1&\approx 2.52326>2.36
	\\\hline
	\end{array}
\]

Consider a simple permutation $\pi$ which does not lie in
\[
	\langle\Av(321)\cup\Av(123)\rangle\cup\{25314, 41352\}
	=
	\langle\langle\Av(321)\rangle\cup\langle\Av(123)\rangle\rangle\cup\{25314, 41352\}.
\]
Such a permutation must contain both a basis element of $\langle\Av(321)\rangle$ and a basis element of $\langle\Av(123)\rangle$. First suppose that $\pi$ contains neither $25314$ nor $41352$. In this case, the basis element of $\langle\Av(321)\rangle$ contained in $\pi$ shows that $\gr(\bigoplus\Sub(\pi))>2.36$, while the basis element of $\langle\Av(123)\rangle$ contained in $\pi$ shows that $\gr(\bigominus\Sub(\pi))>2.36$. Therefore $\pi\notin\G_\gamma$, as desired.

Now suppose that $\pi$ contains $25314$ or its reverse $41352$, and thus $\pi$ properly contains one of these two permutations. Because $\pi$ is simple, any occurrence of one of these permutations must either be \emph{separated}---meaning that there is an entry outside the rectangular hull of its entries which lies either vertically or horizontally amongst these entries---or there must be at least one more entry inside this rectangular hull. As we have already shown that $\pi$ cannot contain a proper inflation of either of these permutations there are, up to symmetry, three cases, shown in Figure~\ref{fig-25314-extensions}. The growth rates of the sum/skew closures for these four cases are shown below. Note that all are greater than $2.36$.

\begin{figure}
\begin{footnotesize}
\begin{center}
	\begin{tikzpicture}[scale=0.5]
		\draw [lightgray, fill=lightgray] (0,1) rectangle (2,3);
		\draw [lightgray, fill=lightgray] (1,4) rectangle (3,6);
		\draw [lightgray, fill=lightgray] (2,2) rectangle (4,4);
		\draw [lightgray, fill=lightgray] (3,0) rectangle (5,2);
		\draw [lightgray, fill=lightgray] (4,3) rectangle (6,5);
		\foreach \i in {1,...,5} {
			\draw [darkgray, ultra thick, line cap=round] (0,\i)--(6,\i);
			\draw [darkgray, ultra thick, line cap=round] (\i,0)--(\i,6);
		}
		\absdot{(1,2)}{};
		\absdot{(2,5)}{};
		\absdot{(3,3)}{};
		\absdot{(4,1)}{};
		\absdot{(5,4)}{};
		% Labels:
		\node at (0.5,4.5) {\mbox{A}};
		\node at (1.5,0.5) {\mbox{A}};
		\node at (5.5,1.5) {\mbox{A}};
		\node at (4.5,5.5) {\mbox{A}};
		\node at (0.5,3.5) {\mbox{B}};
		\node at (2.5,0.5) {\mbox{B}};
		\node at (5.5,2.5) {\mbox{B}};
		\node at (3.5,5.5) {\mbox{B}};
		\node at (1.5,3.5) {\mbox{C}};
		\node at (2.5,1.5) {\mbox{C}};
		\node at (3.5,4.5) {\mbox{C}};
		\node at (4.5,2.5) {\mbox{C}};
	\end{tikzpicture}
\end{center}
\end{footnotesize}
\caption{The three additional types of permutations formed by inserting an entry into $25314$, identified up to symmetry. Here the light gray shading indicates regions which would give rise to a proper an inflation of $25314$, a case we have already eliminated.}
\label{fig-25314-extensions}
\end{figure}
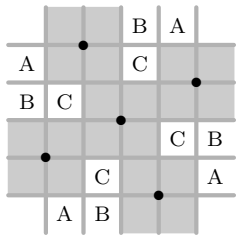

\[
	\begin{array}{clllll}
	\hline
	\mbox{type}&\mbox{extension}&\mbox{sequence of sum}&\mbox{sum closure}&\mbox{sequence of skew}&\mbox{skew closure}
	\\
	&\mbox{of $25314$}&\mbox{indecomposables}&\mbox{growth rate}&\mbox{indecomposables}&\mbox{growth rate}
	\\\hline
	\mbox{A}&526314&1, 1, 3, 9, 6, 1&\approx 2.54894&1, 1, 3, 4, 3, 1&\approx 2.36772
	\\
	\mbox{B}&426315&1, 1, 3, 7, 5, 1&\approx 2.48545&1, 1, 3, 7, 5, 1&\approx 2.48545
	\\
	\mbox{C}&246315&1, 1, 3, 5, 4, 1&\approx 2.41421&1, 1, 3, 7, 5, 1&\approx 2.48545
	\\\hline
	\end{array}
\]
From this chart we see that no simple permutation in $\G_{2.36}$ may properly contain either $25314$ or $41352$, completing the proof.
\end{proof}

% Note: We could have computed the basis of $\langle\Av(321)\cup\Av(123)\rangle\cup\{25314, 41352\}$ and then derived the result from this, but that seems difficult.

%
%
%
%
%

\section{Bounded Alternations and Substitution Depth}

In this section we bound both the length of alternations in and the substitution depth (defined herein) of our cell classes. In fact, the results of this section apply to all cell classes $\G_\gamma$ for $\gamma<1+\varphi$ where $\varphi$ denotes the golden ratio, and thus these results hold up to growth rate approximately $2.61803$. In the result below we say that a class has \emph{bounded alternations} if it does not contain arbitrarily long alternations.

\begin{proposition}
\label{prop-bdd-alts}
The cell class $\G_{\gamma}$ has bounded alternations if and only if $\gamma<1+\varphi$. Conversely, the cell class $\G_{1+\varphi}$ contains all wedge and parallel alternations (of all orientations).
\end{proposition}
\begin{proof}
If a cell class $\G_\gamma$ were to have arbitrarily long alternations then Proposition~\ref{prop-alt-par-wedge} (essentially, the Erd\H{o}s--Szekeres Theorem) would imply that it contains arbitrarily long parallel or wedge alternations. Because classes of the form $\G_{\gamma}$ are closed under all eight symmetries of permutation containment, it suffices to consider parallel alternations oriented as $//$ and wedge alternations oriented as $/\setminus$.

First, a wedge alternation oriented as $/\setminus$ is sum indecomposable if and only if it ends with its least entry. It follows that for $n\ge 2$ there are $2^{n-2}$ sum indecomposable wedge alternations of this orientation, so the generating function for the sum closure of this set is
\[
	\frac{1}{1-\left(x+\frac{x^2}{1-2x}\right)}
	=
	\frac{1-2x}{1+3x-x^2},
\]
and the growth rate of this class is precisely $1+\varphi$. Because the reverses of these permutations are skew indecomposable we see that the skew closure of wedge alternations of this orientation also has growth rate $1+\varphi$. Thus $\G_\gamma$ cannot contain arbitrarily long wedge alternations for $\gamma<1+\varphi$. In the other direction, $\G_{1+\varphi}$ contains all wedge alternations.

Next, a parallel alternation oriented as $//$ is sum indecomposable if and only if it does not begin with its least entry or end with its greatest entry. It follows easily that for $n\ge 2$ there are again precisely $2^{n-2}$ such sum indecomposable parallel alternations of this orientation, so the sum closure of these parallel alternations also has growth rate $1+\varphi$. This implies that these parallel alternations belong to $\G_{1+\varphi}$.

Finally, a parallel alternation of this orientation is skew indecomposable if it either ends with its greatest entry or begins with its smallest entry, from which it follows that there are at least $2^{n-2}$ such skew indecomposable parallel alternations of each length $n\ge 2$. Therefore the growth rate of the skew closure of these parallel alternations is at least $1+\varphi$, implying that $\G_\gamma$ does not contain arbitrarily long parallel alternations for $\gamma\ge 1+\varphi$ and completing the proof.
\end{proof}

Albert, Linton, and Ru\v{s}kuc~\cite{albert:the-insertion-e:} (see also Vatter~\cite{vatter:finding-regular:}) have introduced a length-preserving bijection between permutation classes and formal languages known as the \emph{insertion encoding}. They proved that a permutation class corresponds to a regular language under the insertion encoding if and only if it does not contain arbitrarily long vertical alternations. Thus by Proposition~\ref{prop-bdd-alts}, the classes $\G_{\gamma}$ for $\gamma<1+\varphi$ are in bijection with regular languages and thus have rational generating functions. However, we do not make use of this consequence here.

In \cite{vatter:small-permutati:}, once a $\G_\gamma$-gridding (for $\gamma<\kappa$) of a class was obtained, it was sliced into a $\G_\gamma$-gridding with very restricted structure. Crucial to this procedure was the observation that for $\gamma<\kappa$, $\G_\gamma$ has bounded substitution depth and contains only finitely many simple permutations. For the remainder of this section we review the notion of substitution depth and show that the bounded substitution depth property extends to $\G_\gamma$ for all $\gamma<1+\varphi$. On the other hand, Proposition~\ref{prop-Ggamma-simples} shows that $\G_\xi$ does not have a bounded number of simple permutations and thus the next section is devoted to the development of an alternate approach. Section~\ref{sec-slicing} then details the slicing procedure.

Proposition~\ref{simple-decomp-1} shows that every permutation except $1$ can be expressed as the inflation of a unique simple permutation $\sigma$ of length at least $2$, and that if $\sigma$ is nonmonotone then the intervals of the inflation are also unique. In the case where $\sigma$ is monotone (in which case $\sigma$ is $12$ or $21$), we can instead express the permutation as the sum (resp., skew sum) of a unique sequence of sum (resp., skew) indecomposable permutations (which are themselves intervals). We can represent the process of recursively decomposing the permutation $\pi$ in this manner by means of a rooted tree called its \emph{substitution decomposition tree}. For example, Figure~\ref{fig-subst-tree} shows the substitution decomposition tree associated to our example from Figure~\ref{fig-479832156}.

\begin{figure}
\begin{footnotesize}
\begin{center}
	\begin{tikzpicture}[scale=0.2, baseline=(current bounding box.center)]
		% The 1:
		\draw [lightgray, fill] (1,4) circle (20pt);
		% The 132:
		\draw[lightgray, fill, rotate around={-45:(2.9,7.9)}] (2.9,7.9) ellipse (45pt and 55pt);
		% The bottom 321:
		\draw[lightgray, fill, rotate around={45:(6,2)}] (6,2) ellipse (25pt and 60pt);
		% Right most 12:
		\draw[lightgray, fill, rotate around={-45:(8.5,5.5)}] (8.5,5.5) ellipse (20pt and 40pt);
		% The permutation:
		\plotpermbox{0.5}{0.5}{9.5}{9.5};
		\plotperm{4,7,9,8,3,2,1,5,6};
	\end{tikzpicture}
\quad\quad
	\begin{tikzpicture}[baseline=(current bounding box.center)]
	\tikzset{level distance=18pt}
	\Tree	[.$2413$	$1$
				[.$\bigoplus$ $1$ [.$\bigominus$ $1$ $1$ ] ]
				[.$\bigominus$ $1$ $1$ $1$ ]
				[.$\bigoplus$ $1$ $1$ ]
			]
	\end{tikzpicture}
\quad\quad
	\begin{tikzpicture}[baseline=(current bounding box.center)]
	\tikzset{level distance=18pt}
	\Tree	[.$\bigoplus$	$1$
				[.$\bigominus$	$1$	$\triangle$ ]
			]
	\end{tikzpicture}
\end{center}
\end{footnotesize}
\caption{The plot of the permutation $479832156$ (left), its substitution decomposition tree (center), and the substitution decomposition tree of a wedge alternation oriented as $>$ (right); here the $\triangle$ symbol indicates that the tree repeats in the same pattern.}
\label{fig-subst-tree}
\end{figure}

The \emph{substitution depth} of $\pi$ is the height of its substitution decomposition tree, so for example, the substitution depth of the permutation shown on the left of Figure~\ref{fig-subst-tree} is $3$, while the substitution depth of every simple or monotone permutation is $1$. The substitution tree of a wedge alternation is shown on the right of Figure~\ref{fig-subst-tree}. As substitution depth is closed under symmetries of the square, this figure indicates that the substitution depth of wedge alternations is unbounded. Indeed, we show below that they are the only obstructions to bounded substitution depth. A proof of this result was given in \cite{vatter:small-permutati:}, though we present a simpler proof below.

\begin{proposition}[Vatter~{\cite[Proposition~4.2]{vatter:small-permutati:}}]
\label{prop-long-path-wedge}
If the permutation $\pi$ has substitution depth at least $8n$ then $\pi$ contains a wedge alternation of length at least $n$.
\end{proposition}
\begin{proof}
There are four possible orientations of wedge alternations which we denote by $\wedge$, $>$, $\vee$, and $<$. Label these orientations $1$ to $4$ respectively. We claim that if the permutation $\pi$ has substitution depth greater than $2\sum k_i$ then for some $i\in\{1,2,3,4\}$, $\pi$ contains a wedge alternation of orientation $i$ of length at least $k_i$. The proof is by induction on $\sum k_i$. For the base case(s), the claim is trivial if any $k_i$ equals $1$, so we may assume that each $k_i\ge 2$ for all $i$.

Let $\pi$ be a permutation of substitution depth at least $2\sum k_i$. Therefore $\pi$ contains an interval $\alpha$ of substitution depth at least $2\sum k_1-1$. We see similarly that $\alpha$ contains an interval $\delta$ of substitution depth at least $2\left(\sum k_1-1\right)$. By symmetry we may assume that $\pi$ is either a sum or the inflation of a simple permutation of length at least four. This implies that $\alpha$ is either a skew sum, the inflation of a simple permutation of length at least four, or, if $\pi$ is the inflation of a simple permutation of length at least four, $\alpha$ could be a sum. In the last case we instead consider the reverse of $\pi$ and thus we may always assume that $\alpha$ is either a skew sum or the inflation of a simple permutation of length at least four.

\begin{figure}
\[
	\begin{array}{ccccccc}
		\begin{tikzpicture}[scale=0.2]
			\plotpermbox{1}{1}{7}{7};
			\plotpermbox{2}{2}{6}{6};
			\plotpermbox{3}{3}{5}{5};
			\plotpartialperm{1/1,6/2};
			\node at (4,4) {$\delta$};
		\end{tikzpicture}
	&\quad&
		\begin{tikzpicture}[scale=0.2]
			\plotpermbox{1}{1}{7}{7};
			\plotpermbox{2}{2}{6}{6};
			\plotpermbox{3}{3}{5}{5};
			\plotpartialperm{1/1,2/6};
			\node at (4,4) {$\delta$};
		\end{tikzpicture}
	&\quad&
		\begin{tikzpicture}[scale=0.2]
			\plotpermbox{1}{1}{7}{7};
			\plotpermbox{2}{2}{6}{6};
			\plotpermbox{3}{3}{5}{5};
			\plotpartialperm{7/7,6/2};
			\node at (4,4) {$\delta$};
		\end{tikzpicture}
	&\quad&
		\begin{tikzpicture}[scale=0.2]
			\plotpermbox{1}{1}{7}{7};
			\plotpermbox{2}{2}{6}{6};
			\plotpermbox{3}{3}{5}{5};
			\plotpartialperm{7/7,2/6};
			\node at (4,4) {$\delta$};
		\end{tikzpicture}
	\end{array}
\]
\caption{The four cases in the proof of Proposition~\ref{prop-long-path-wedge}.}
\label{fig-long-path-wedge}
\end{figure}
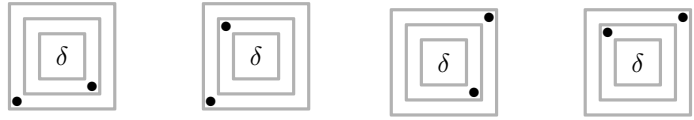

This implies that $\pi$ contains an entry lying either to the southwest or northeast of $\alpha$ and that $\alpha$ contains an entry lying either to the southeast or northwest of $\delta$. These four cases are depicted in Figure~\ref{fig-long-path-wedge}. As can be seen by inspecting each of these cases, at least one of the two entries outside $\delta$ can be used to extend the longest wedge alternation of at least one of the four orientations contained in $\delta$. For example, consider the first case shown in Figure~\ref{fig-long-path-wedge}. Because the substitution depth of $\delta$ is at least $2\left(\sum k_i-1\right)$, we see by induction that $\delta$ either contains a wedge alternation of length at least $k_1-1$ oriented as $\wedge$, in which case $\pi$ contains a wedge alternation of length at least $k_1$ of this orientation, or $\delta$ contains a wedge alternation of length at least $k_i$ of one of the orientations $>$, $\vee$, or $<$, in which case we are immediately done.
\end{proof}

We define the \emph{substitution depth} of a class $\C$ as the maximum substitution depth of any of its members; note that his quantity may be infinite. In fact Proposition~\ref{prop-long-path-wedge} and our previous comments show that the substitution depth of $\C$ is infinite if and only if $\C$ contains arbitrarily long wedge alternations. Therefore from the previous two results of this section we obtain the following.

\begin{proposition}
\label{prop-bdd-subst-depth}
The class $\G_\gamma$ has bounded substitution depth if and only if $\gamma<1+\varphi$.
\end{proposition}

\section{Concentration}
\label{sec-concentration}

In this section we introduce and study a new property of permutation classes, that of concentration. This property allows us to apply the techniques of \cite{vatter:small-permutati:} even when the cell class $\G_\gamma$ contains infinitely many simple permutations (as it does for $\gamma\ge\kappa$ by Proposition~\ref{prop-Ggamma-simples}) and represents one of the most significant new ideas in this work.

For natural numbers $q$ and $r$, we say that the permutation $\pi$ is \emph{$(q,r)$-concentrated} if given any horizontal or vertical line $L$ through the plot of $\pi$ (we assume throughout that our lines never pass through entries of the permutation), all but at most $r$ of the entries of $\pi$ can be covered by $q$ independent axis-parallel rectangles which do not intersect $L$ (here $q$ and $r$ are meant to invoke the terms ``quotient'' and ``remainder''). We further say that the set $X$ of permutations is $(q,r)$-concentrated if every $\pi\in X$ is $(q,r)$-concentrated and that $X$ is \emph{concentrated} if it is $(q,r)$-concentrated for some values of $q$ and $r$. Note that if $X$ and $Y$ are both concentrated sets of permutations then $X\cup Y$ is also concentrated.

For example, the class $\Av(21)\cup\Av(12)$ of monotone permutations is clearly $(2,0)$-concentrated (in fact, it is the largest class which is $(2,0)$-concentrated), while the class $\O$ of oscillations is $(2,2)$-concentrated as indicated by Figure~\ref{fig-incosc-concentration}. The following result, together with Propositions~\ref{prop-Ggamma-simples} and \ref{prop-bdd-subst-depth}, implies that the classes $\G_\gamma$ for $\gamma<\kappa$ are concentrated.

\begin{proposition}
\label{prop-fin-simples-concentration}
If the longest simple permutation in the class $\C$ has length $s$ and the substitution depth of $\C$ is $d$ then $\C$ is $(sd,0)$-concentrated.
\end{proposition}
\begin{proof}
The case $s=1$ is trivial (as then $\C=\{1\}$, which is $(1,0)$-concentrated), so we may assume that $s\ge 2$. We prove the claim by induction on $d$. If $d=1$ then every member of $\C$ is either a simple permutation of length at most $s$ (in fact, because $\C$ is a class, this possibility can be ruled out, but we do not need to) or a monotone permutation. In either case, because $s\ge 2$, $\C$ is $(s,0)$-concentrated.

Suppose now that $d\ge 2$, take $\pi\in\C$, and let $L$ be an arbitrary horizontal or vertical line through the plot of $\pi$. Write $\pi=\sigma[\alpha_1,\dots,\alpha_m]$ where $\sigma$ is either simple or monotone and each $\alpha_i$ has substitution depth at most $d-1$. The line $L$ can pass through at most one $\alpha_i$, and by moving $L$ slightly (without affecting its position relative to the entries of $\pi$) we may assume that it passes through precisely one interval $\alpha_i$. By induction, the entries of $\alpha_i$ can be partitioned into $s(d-1)$ independent axis-parallel rectangles which do not intersect $L$. If $\sigma$ is simple we take the other intervals, $\alpha_j$ for $j\neq i$, to be our other axis-parallel rectangles, showing that $\pi$ is $(s(d-1)+s-1,0)$-concentrated. If $\sigma$ is monotone, we use two additional rectangles, one containing $\alpha_1$, $\dots$, $\alpha_{i-1}$ and the other containing $\alpha_{i+1}$, $\dots$, $\alpha_m$, showing that $\pi$ is $(s(d-1)+2,0)$-concentrated. Thus in either case we see that $\pi$ is $(sd,0)$-concentrated, completing the proof.
\end{proof}

\begin{figure}
\begin{footnotesize}
\begin{center}
	\begin{tikzpicture}[scale=0.4]
		% rectangles:
		\draw [thick, darkgray, fill=lightgray, line cap=round] (0.75,0.75) rectangle (5.25,6.25);
		\draw [thick, darkgray, fill=lightgray, line cap=round] (6.75,6.75) rectangle (10.25,10.25);
		% points:
		\absdot{(1,2)}{};
		\absdot{(3,4)}{};
		\absdot{(4,1)}{};
		\absdot{(5,6)}{};
		\absdot{(6,3)}{};
		\absdot{(7,8)}{};
		\absdot{(8,5)}{};
		\absdot{(9,10)}{};
		\absdot{(10,7)}{};
		% slice line:
		\draw [dashed, thick, line cap=round] (5.5,0)--(5.5,11);
		\draw [dotted, thick, line cap=round] (0,6.5)--(11,6.5);
		% border
		\draw [lightgray, ultra thick, line cap=round] (0,0) rectangle (11,11);
	\end{tikzpicture}
\quad\quad
	\begin{tikzpicture}[scale=0.4]
		% rectangles:
%		\draw [thick, darkgray, fill=lightgray, line cap=round] (0.75,0.75) rectangle (5.25,6.25);
%		\draw [thick, darkgray, fill=lightgray, line cap=round] (6.75,6.75) rectangle (10.25,10.25);
		% points:
		\draw[thick] (0.75,1.75)--(1.25,2.25);
		\draw[thick] (2.75,3.75)--(3.25,4.25);
		\draw[thick] (3.75,0.75)--(4.25,1.25);
		\draw[thick] (4.75,5.75)--(5.25,6.25);
		\draw[thick] (5.75,2.75)--(6.25,3.25);
		\draw[thick] (6.75,7.75)--(7.25,8.25);
		\draw[thick] (7.75,4.75)--(8.25,5.25);
		\draw[thick] (8.75,9.75)--(9.25,10.25);
		\draw[thick] (9.75,6.75)--(10.25,7.25);
		% slice line:
		\draw [dashed, thick, line cap=round] (5.5,0)--(5.5,11);
		\draw [dotted, thick, line cap=round] (0,6.5)--(11,6.5);
		% border
		\draw [lightgray, ultra thick, line cap=round] (0,0) rectangle (11,11);
	\end{tikzpicture}
\end{center}
\end{footnotesize}
\caption{On the left, dividing an increasing oscillation into two independent rectangles and two extra entries; here the dotted line is $L$. The figure on the right shows that when the increasing oscillation sequence is inflated by $\Av(21)$, it is not concentrated.}
\label{fig-incosc-concentration}
\end{figure}
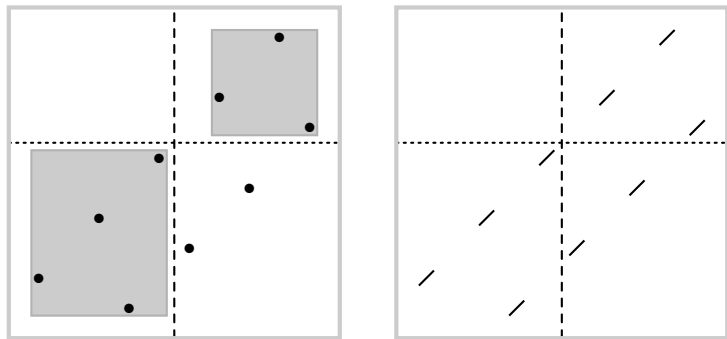

Ideally, we would characterize the concentrated permutation classes. As a foray in this direction, note that (as suggested by Figure~\ref{fig-alternation-oscillation}) a concentrated class cannot contain arbitrary long alternations. (This in turn implies via Proposition~\ref{prop-long-path-wedge} that concentrated classes must have bounded substitution depth.) Thus one might hope that a class is concentrated if and only if it has bounded alternations. However, the inflation of increasing oscillations by the increasing class, $\O_I[\Av(21)]$, has bounded alternations but fails to be concentrated, as indicated on the right of Figure~\ref{fig-incosc-concentration}. This example, which is generalized below, shows that it may be difficult to demarcate the precise boundary between concentrated and non-concentrated classes.

\begin{proposition}
Suppose that the class $\C$ contains infinitely many simple permutations and that the class $\D$ is infinite. Then the class $\C[\D]$ is not concentrated.
\end{proposition}
\begin{proof}
By Theorem~\ref{thm-alt-or-osc} we see that $\C$ contains arbitrarily long alternations or oscillations. If $\C$ contains arbitrarily long alternations then it is not concentrated by our previous remarks, so $\C[\D]$ is not concentrated either. If $\C$ contains arbitrarily long oscillations then it either contains all increasing oscillations (and thus the class $\O_I$) or all decreasing oscillations (and thus the class $\O_D$). Moreover, because $\D$ is infinite, the Erd\H{o}s--Szekeres Theorem shows that it must contain the class of increasing permutations or the class of decreasing permutations, i.e., $\Av(21)$ or $\Av(12)$. Therefore $\C[\D]$ must contain one of the classes $\O_I[\Av(21)]$, $\O_I[\Av(12)]$, $\O_D[\Av(21)]$, or $\O_D[\Av(12)]$, and none of these classes is concentrated.
\end{proof}

Therefore we leave the following problem open.

\begin{problem}
\label{prob-characterize-concentration}
Characterize the concentrated permutation classes.	
\end{problem}

It may be easier to characterize the concentrated \emph{cell classes}. In this context we conjecture that the threshold for concentration is the growth rate $1+\varphi$. Note that this would be best possible because $\G_{1+\varphi}$ contains arbitrarily long alternations (Proposition~\ref{prop-bdd-alts}) and thus is not concentrated.

\begin{conjecture}
\label{conj-concentrated-phi}
The cell class $\G_\gamma$ is concentrated if and only if $\gamma<1+\varphi\approx 2.62$.
\end{conjecture}

We settle for showing that the cell class $\G_\xi$ is concentrated, and in order to do so we must consider the geometric structure of $321$-avoiding permutations. Recall from the introduction that given any figure in the plane and permutation $\pi$, we say that $\pi$ can be drawn on the figure if we can choose $n$ points in the figure, no two on a common horizontal or vertical line, label them $1$ to $n$ from bottom to top and then read them from left to right to obtain $\pi$. The following result was first proved by Waton in his thesis~\cite{waton:on-permutation-:} while another proof was given by Albert, Linton, Ru\v{s}kuc, Vatter, and Waton~\cite{albert:on-convex-permu:}.

\begin{proposition}
\label{prop-321-lines}
The class of permutations that can be drawn on any two parallel lines of positive slope is $\Av(321)$.
\end{proposition}

We now establish that the simple permutations in $\G_{2.36}$ are concentrated (we consider their inflations separately).

\begin{proposition}
\label{prop-SiGxi-concentration}
The set of simple permutations in the cell class $\G_{2.36}$ is concentrated.
\end{proposition}
\begin{proof}
By Proposition~\ref{prop-bdd-alts} we may assume that $\G_{2.36}$ contains no parallel alternations of length $2a$ or greater. By symmetry and Proposition~\ref{prop-Gxi-first}, it suffices to consider $25314$ and the simple permutations of $\Av(321)$. As $25314$ is trivially $(5,0)$-concentrated, we move on to considering an arbitrary simple permutation $\sigma\in\Av(321)$. Draw $\sigma$ on two parallel lines of positive slope as shown in Figure~\ref{fig-Gxi-diag-simple} and consider an arbitrary horizontal or vertical line $L$.

By symmetry we may assume that $L$ is horizontal, as represented by the dashed line in Figure~\ref{fig-Gxi-diag-simple}. Now consider the perpendicular line (a dotted line in the figure) which intersects $L$ precisely where $L$ intersects the higher of the two parallel lines (Figure~\ref{fig-Gxi-diag-simple} shows that we have another choice for this line, though we only need one). We seek to show that the northeastern and southwestern rectangles in this figure together contain all but a bounded number of entries of $\sigma$.

\begin{figure}
\begin{footnotesize}
\begin{center}
	\begin{tikzpicture}[scale=1]
		% 321 lines:
		\draw [thick, line cap=round] (0,0.5)--(2.5,3);
		\draw [thick, line cap=round] (0.5,0)--(3,2.5);
		% slice lines:
		\draw [dashed, thick, line cap=round] (0,1.4)--(3,1.4);
		\node [right] at (3,1.4) {$L$};
		\draw [dotted, thick, line cap=round] (0.9,0)--(0.9,3);
		% border
		\draw [lightgray, ultra thick, line cap=round] (0,0) rectangle (3,3);
		% caption
		\node at (1.95,0.5) {\mbox{bounded $\#$}};
		\node at (1.95,0.25) {\mbox{of entries}};
	\end{tikzpicture}
\quad\quad
	\begin{tikzpicture}[scale=1]
		% 321 lines:
		\draw [thick, line cap=round] (0,0.5)--(2.5,3);
		\draw [thick, line cap=round] (0.5,0)--(3,2.5);
		% slice lines:
		\draw [dashed, thick, line cap=round] (0,1.4)--(3,1.4);
		\node [right] at (3,1.4) {$L$};
		\draw [dotted, thick, line cap=round] (1.9,0)--(1.9,3);
		% border
		\draw [lightgray, ultra thick, line cap=round] (0,0) rectangle (3,3);
		% caption
		\node at (0.95,2.75) {\mbox{bounded $\#$}};
		\node at (0.95,2.5) {\mbox{of entries}};
	\end{tikzpicture}
\end{center}
\end{footnotesize}
\caption{Two ways to establish concentration for simple permutations in $\Av(321)$. In these figures the dashed line is $L$ while the dotted line shows how to divide the figure into two independent rectangles and a small dependent region.}
\label{fig-Gxi-diag-simple}
\end{figure}
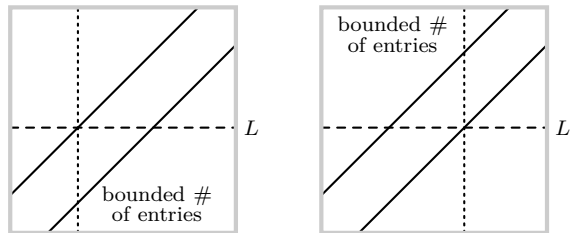

There are no points in the northwestern region by construction, so it suffices to bound the number of points in the southeastern region. Because $\sigma$ is simple, every consecutive pair of points in this corner of the gridding must be separated by a point in the region above it or in the region to its left. If $a+1$ or more consecutive pairs were separated in the same direction then $\sigma$ would contain a parallel alternation of length $2a$, a contradiction. Therefore there can be at most $2a+1$ entries in the southeastern region, proving that the simple permutations in $\G_{2.36}$ are $(2,2a+1)$-concentrated.
\end{proof}

We now establish the main result of this section. The proof of the following result shows that it holds for $\G_{2.309}$ (which contains $\G_\xi$) and we believe that a more detailed argument could establish the result for $\G_{2.36}$. As in the proof of the previous proposition, we make no effort to optimize the concentration parameters.

\begin{theorem}
\label{thm-Gxi-concentration}
The cell class $\G_\xi$ is concentrated.
\end{theorem}
\begin{proof}
By Proposition~\ref{prop-bdd-alts}, $\G_\xi$ has bounded alternations and thus Theorem~\ref{thm-alt-or-osc} implies that there is an integer $s$ such that every simple permutation in $\G_\xi$ of length greater than $s$ contains an oscillation of length at least $7$. Let $d$ denote the substitution depth of $\G_\xi$, which is bounded by Proposition~\ref{prop-bdd-subst-depth}. As the set of inflations of simple permutations of length at most $s$ with substitution depth at most $d$ is concentrated by Proposition~\ref{prop-fin-simples-concentration}, it suffices to consider simple permutations of length greater than $s$ and their inflations. By Proposition~\ref{prop-SiGxi-concentration} we may assume that the simple permutations themselves are $(q,r)$-concentrated for some integers $q$ and $r$.

Take $\pi\in\G_\xi$ to be an inflation of a simple permutation of length greater than $s$. By symmetry we may assume that $\pi$ lies in $\G_\xi$ because $\gr(\bigoplus\Sub(\pi))<\xi$. By our choice of $s$ we know that $\pi$ contains an oscillation (either increasing or decreasing) of length at least $7$. The two decreasing oscillations of length $5$ (at least one of which is contained in all longer decreasing oscillations) are $35142$ and $42531$, and the chart below shows that these cannot be contained in $\pi$, so we may assume that $\pi$ contains an increasing oscillation of length at least $7$.
\[
	\begin{array}{lll}
	\hline
	\mbox{decreasing}&\mbox{sequence of sum}&\mbox{sum closure}
	\\
	\mbox{oscillation}&\mbox{indecomposables}&\mbox{growth rate}
	\\\hline
	35142&1,1,3,5,1&\approx 2.36691>\xi
	\\
	42531&1,1,3,5,1&\approx 2.36691>\xi
	\\\hline
	\end{array}
\]
Because $\pi$ contains an increasing oscillation of length at least $7$, $\Sub(\pi)$ also contains all increasing oscillations of length at most $6$. This implies that the sequence of sum indecomposable members of $\Sub(\pi)$ is term-wise at least $1,1,2,2,2,2,1$.

We claim that every proper interval in $\pi$ is a copy of $12$. Before proving this claim we demonstrate how it finishes the proof. Take $\pi\in\G_\xi$ and suppose that $\pi$ is the inflation of the simple permutation $\sigma$. Now consider a vertical or horizontal line $L$. If $L$ does not slice through an interval of $\pi$ then $\pi$ may be partitioned into axis-parallel rectangles together with remainder entries just as $\sigma$ would be, though the remainder entries may be inflated by copies of $12$, so all but $2r$ entries of $\pi$ can be covered by $q$ independent axis-parallel rectangles which do not intersect $L$. On the other hand if $L$ slices through an interval of $\pi$ then we can partition the entries of $\pi$ which do lie in this interval as before and consider the two entries of this interval as additional remainder entries. In either case, $\pi$ is $(q,2r+2)$-concentrated.

We now prove the claim. Suppose first that $\pi$ contains a $21$ interval. Because $\G_\xi$ does not contain inflations of $25314$ or $41352$, the inflated entry must be part of a copy of $2413$ or $3142$. In either case it is easy to see that $\Sub(\pi)$ contains sum indecomposable permutations of lengths three and four which are not increasing oscillations (one is $321$ while the other depends on which entry is inflated). Therefore the sequence of sum indecomposable members of $\Sub(\pi)$ is term-wise at least $1,1,3,3,2,2,1$. This implies that $\gr(\bigoplus\Sub(\pi))>2.32638>\xi$, a contradiction. Next suppose that $\pi$ contains a $123$ interval. By the same reasoning as the previous case the inflated entry must be part of a copy of $2413$ or $3142$. In this case, it is easy to see that $\Sub(\pi)$ contains two sum indecomposable permutations of lengths four and five which are not increasing oscillations (corresponding to taking either two or three entries of the interval) and thus the sequence of sum indecomposable members of $\Sub(\pi)$ is at least $1,1,2,4,4,2,1$. This shows that $\gr(\bigoplus\Sub(\pi))>2.30985>\xi$, another contradiction, completing the proof of the claim and thus of the theorem as well.
\end{proof}

We conclude with a result which allows us to slice concentrated classes by multiple lines at once.

\begin{proposition}
\label{prop-many-lines}
Suppose that the class $\C$ is $(q,r)$-concentrated and $\pi\in\C$. Given a collection $\mathfrak{L}$ of $\ell$ horizontal or vertical lines through the plot of $\pi$, all but at most $r\ell$ entries of $\pi$ can be covered by $(q-1)\ell+1$ independent axis-parallel rectangles which do not intersect any line in $\mathfrak{L}$.
\end{proposition}
\begin{proof}
We prove the proposition by induction on $\ell$. The case of $\ell=1$ is the definition of $(q,r)$-concentration, so assume that $\ell\ge 2$ and that the claim holds for $\ell-1$ lines.

Let $\mathfrak{L}$ be a collection of $\ell$ horizontal and vertical lines and choose an arbitrary line $L\in\mathfrak{L}$. By induction, all but at most $r(\ell-1)$ entries of $\pi$ can be covered by $(q-1)(\ell-1)+1$ independent axis-parallel rectangles which do not intersect any line in $\mathfrak{L}\setminus\{L\}$. We may assume that $L$ intersects precisely one of these rectangles, say $R$, because it cannot intersect more than one and if it does not intersect any of the rectangles then we are done. Because $\C$ is a class, the entries covered by $R$ are order isomorphic to a member of $\C$ and thus all but $r$ of them can be covered by at most $q$ axis-parallel rectangles not intersecting $L$. Combining these rectangles with our previous decomposition increases the number of entries not covered by $r$ and the number of rectangles by $q-1$ (we no longer need $R$ itself), establishing the bound.
\end{proof}

\section{Slicing}
\label{sec-slicing}

Given a class $\C$ with $\ugr(\C)<\xi$, we have established that $\C\subseteq\Grid(\M)$ for a matrix $\M$ whose entries are all equal to some class $\D\subseteq\G_\xi$. Moreover, we have significantly restricted the structure of $\D$. In particular, we have shown that $\D$ has bounded alternations and substitution depth (Propositions~\ref{prop-bdd-alts} and \ref{prop-bdd-subst-depth}) and also that $\D$ is concentrated (Theorem~\ref{thm-Gxi-concentration}).

In this section we restrict the griddings of members of $\C$, and to do this we must be a bit more precise. For any matrix $\M$ of permutation classes an \emph{$\M$-gridded permutation} is a permutation $\pi\in\Grid(\M)$ equipped with an $\M$-gridding. We denote gridded permutations by $\pi^\gridded$ (the matrix will always be clear from context) and we denote the collection of all $\M$-gridded permutations by $\Grid^\gridded(\M)$.

Every permutation in $\Grid(\M)$ corresponds to at least one gridded permutation in $\Grid^\gridded(\M)$. Moreover, if $\M$ is a $t\times u$ matrix then every permutation of length $n$ in $\Grid(\M)$ corresponds to at most ${n+t-2\choose t-1}{n+u-2\choose u-1}$ gridded permutations in $\Grid^\gridded(\M)$, because there are only so many places one can insert grid lines. Thus we obtain the following.

\begin{observation}
\label{obs-gridded-gr}
For a matrix $\M$ of permutation classes and a class $\C\subseteq\Grid(\M)$, the upper (resp., lower) growth rate of $\C$ is equal to the upper (resp., lower) growth rate of the sequence enumerating $\M$-griddings of members of $\C$.
\end{observation}

Our next two results establish that gridded permutations $\pi^\gridded$ in classes with growth rates less than $\xi$ cannot contain large independent sets of certain types of rectangles. Here the notion of independence is the same as in Section~\ref{sec-gridding}---two rectangles are independent if both their $x$- and $y$-axis projections are disjoint. In fact, the main result of this section, Theorem~\ref{thm-slicing}, follows from Gy{\'a}rf{\'a}s and Lehel's Theorem~\ref{rectangles-lemma} just as Theorem~\ref{thm-gridding-characterization} did.

Let $\pi^\gridded$ denote any gridded permutation. A \emph{separated nonmonotone rectangle} is an axis-parallel rectangle $R$ contained in the plot of $\pi$ such that
\begin{itemize}
\item $R$ is fully contained in one cell of $\pi^\gridded$,
\item $\pi(R)$ is nonmonotone (and thus contains both $12$ and $21$), and
\item $\pi^\gridded$ has an entry in another cell which separates $R$ either vertically or horizontally.
\end{itemize}
An example is shown on the left of Figure~\ref{fig-SN-rectangle}. Note that we do not insist that the separating entry actually separate the entries of $\pi(R)$; it need only separate the rectangle $R$ itself. Our next result shows that we can bound independent sets of separated nonmonotone rectangles in classes with growth rates under $1+\sqrt{2}$. Importantly, there are no hypotheses on the classes which occur in the matrix $\M$.

\begin{proposition}
\label{prop-sep-nonmono-rects}
Let $\C\subseteq\Grid(\M)$. If there is no bound on the size of independent sets of separated nonmonotone rectangles in $\M$-gridded members of $\C$ then $\lgr(\C)\ge 1+\sqrt{2}\approx 2.41421$.	
\end{proposition}
\begin{proof}
Suppose that $\M$ is a $t\times u$ matrix of permutation classes and that $\C\subseteq\Grid(\M)$ contains gridded permutations with at least $64tum^4$ independent separated nonmonotone rectangles for every value of $m$. Fix a value of $m$ and take $\pi^\gridded$ to be such a permutation. At least $64m^4$ of the separated nonmonotone rectangles must lie in the same cell of $\pi^\gridded$. Furthermore, at least $16m^4$ of these separated nonmonotone rectangles are separated in the same direction (left, right, up, or down). Because we can treat independent rectangles as permutations themselves, the Erd\H{o}s--Szekeres Theorem therefore implies that some cell of $\pi^\gridded$ contains an increasing or decreasing sequence containing at least $4m^2$ independent separated nonmonotone rectangles, each separated in the same direction. We may then apply the Erd\H{o}s--Szekeres Theorem to the separating entries themselves to see that some cell of $\pi^\gridded$ contains a monotone sequence of at least $2m$ independent separated nonmonotone rectangles each separated in the same direction by a monotone sequence of separating entries.

\begin{figure}
\begin{center}
	\begin{tikzpicture}[scale=0.2, baseline=(current bounding box.center)]
		\draw [darkgray, fill=lightgray, ultra thick, line cap=round] (2,1.5) rectangle (7,8);
		\plotpartialperm{3/4,5/6,6/2,10/7};
		\plotpermbox{1}{1}{8}{11};
		\plotpermbox{9}{1}{11}{11};
		\draw (10,7)--(4,7);
	\end{tikzpicture}
\quad\quad
	\begin{tikzpicture}[scale=0.2, baseline=(current bounding box.center)]
		\plotperm{2,1,5,4,8,7,11,10,3,6,9};
		\plotpermbox{1}{1}{8}{11};
		\plotpermbox{9}{1}{11}{11};
	\end{tikzpicture}
\quad\quad
	\begin{tikzpicture}[scale=0.2, baseline=(current bounding box.center)]
		\plotperm{2,1,5,4,8,7,11,10,9,6,3};
		\plotpermbox{1}{1}{8}{11};
		\plotpermbox{9}{1}{11}{11};
	\end{tikzpicture}
\end{center}
\caption{From left to right, a separated nonmonotone rectangle and two examples of unavoidable structures arising from arbitrarily large independent sets of separated nonmonotone rectangles.}
\label{fig-SN-rectangle}
\end{figure}
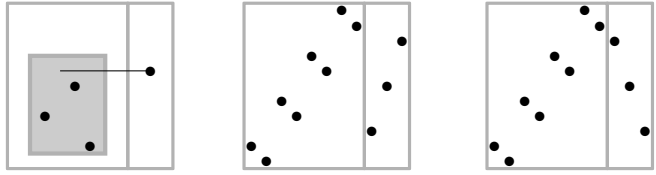

By symmetry, we may assume that this monotone sequence of separated nonmonotone rectangles is increasing and that the separating entries lie to the right of each of the rectangles. Label the rectangles $R_1$, $R_2$, $\dots$ in increasing order. Because every separated nonmonotone rectangle contains both $12$ and $21$, we may choose a copy of $21$ from each odd-indexed rectangle and the separating entry of every even-indexed rectangle. By doing so we obtain a permutation of length $3m$ of the form shown in the center or right of Figure~\ref{fig-SN-rectangle}.

As we can do this for arbitrary $m$, it follows that $\C$ contains the downward closure of all permutations of one of these two forms (up to symmetry). It remains only to show that the growth rates of both downward closures are at least $1+\sqrt{2}$; this is done in \cite[Proposition~A.14]{vatter:small-permutati:} but we provide the brief details for completeness. By Proposition~\ref{obs-gridded-gr}, we may instead enumerate the gridded permutations in these downward closures (where we view the downward closures as being contained in $2\times 1$ grid classes). To encode such permutations we simply read from bottom to top, encoding a single entry on the left by $\l_1$, a copy of $21$ on the left by $\l_2$, and an entry on the right by $\r$. It follows that the set of gridded permutations in either of these classes is bijection with the set of words over the alphabet $\{\l_1,\l_2,\r\}$. Moreover, the generating function for these words (where $\l_1$ and $\r$ contribute $x$ and $\l_2$ contributes $x^2$) is $1/(1-2x-x^2)$, which has growth rate $1+\sqrt{2}$, completing the proof.
\end{proof}

We now have the tools to prove our major slicing result. It may be viewed as a generalization of \cite[Theorem~5.4]{vatter:small-permutati:}, though both the conclusion and the proof technique are different. Note that by Proposition~\ref{prop-sep-nonmono-rects}, the hypotheses of this theorem hold for all classes $\C$ which are $\G$-griddable for a concentrated class $\G$ and satisfy $\lgr(\C)<1+\sqrt{2}$.

\begin{theorem}
\label{thm-slicing}
Suppose that the permutation class $\C\subseteq\Grid(\M)$ for a matrix $\M$ whose entries are all equal to a concentrated class $\G$. If there is a bound on the size of independent sets of separated nonmonotone rectangles in $\M$-gridded members of $\C$ then there is an integer $p$ such that $\C\subseteq\Grid(\M')^{+p}$ for a matrix $\M'$ whose nonempty entries are equal to either $\G$ or a monotone class and in which every nonempty entry that shares a row or column with another nonempty entry is monotone.
\end{theorem}
\begin{proof}
Suppose that $\M$ is a $t\times u$ matrix. Let $\pi\in\C$ be arbitrary and fix an $\M$-gridding $\pi^\gridded$ of it. Let $\mathfrak{R}_{\pi^\gridded}$ denote the set of all separated nonmonotone rectangles of $\pi^\gridded$. Our hypotheses ensure that that there is a constant $m$ depending only on $\C$ so that $\mathfrak{R}_{\pi^\gridded}$ has no independent set of size greater than $m$. Therefore by Theorem~\ref{rectangles-lemma} there is a set $\mathfrak{L}_{\pi^\gridded}$ of at most $f(m)$ vertical and horizontal lines which slice every separated nonmonotone rectangle in $\mathfrak{R}_{\pi^\gridded}$.

Choose integers $q$ and $r$ so that $\G$ is $(q,r)$-concentrated ($\G$ is concentrated by our hypotheses). Each cell of $\pi^\gridded$ is sliced by at most $f(m)$ lines from $\mathfrak{L}_{\pi^\gridded}$, so by Proposition~\ref{prop-many-lines}, all but at most $r\cdot f(m)$ entries of each cell of $\pi^\gridded$ may be covered by $(q-1)f(m)+1$ independent axis-parallel rectangles which do not intersect any line in $\mathfrak{L}_{\pi^\gridded}$. Note that this independence holds only for rectangles arising from the same cell; two rectangles from different cells of the original gridding $\pi^\gridded$ may still be dependent. Also note that in total, at most $tur\cdot f(m)$ entries of $\pi$ are uncovered by these axis-parallel rectangles.

We now convert these rectangles into a new gridding of $\pi$. Given any axis-parallel rectangle produced by the previous slicing operation we extend its four sides to vertical and horizontal lines slicing the entire plot of $\pi$. Doing so for every axis-parallel rectangle produced in every cell of $\pi^\gridded$, we obtain a set $\mathfrak{L}_{\pi^{\gridded}}'$ of at most $4tu\left((q-1)f(m)+1\right)$ lines. These lines slice the plot of $\pi$ into a gridding of bounded size\footnote{The explicit bound is $4tu\left((q-1)f(m)+1\right)+1\times 4tu\left((q-1)f(m)+1\right)+1$, though no effort has been made at optimization.}. Moreover, ignoring the (bounded number of) uncovered entries, in the refined gridding induced by the lines of $\mathfrak{L}_{\pi^\gridded}'$, no cell which contains a nonmonotone subpermutation may lie in the same row or column as another nonempty cell. This is because if those two cells arose from different cells in the original gridding $\pi^{\gridded}$ then they would form a nonmonotone rectangle, which should have been sliced by some line in $\mathfrak{L}_{\pi^{\gridded}}$, while if they arose from the same cell in $\pi^{\gridded}$, they would be independent.

We now create a matrix $\M_{\pi^{\gridded}}$ so that the covered entries of $\pi$ lie in $\Grid(\M_{\pi^\gridded})$. To do so we simply record the contents of each cell of the refined gridding of $\pi$ as either $\G$, if the cell has nonmonotone content, $\Av(21)$ or $\Av(12)$, if the cell has monotone context, or $\emptyset$, if the cell contains no entries. Setting $p=tur\cdot f(m)$, the maximum number of uncovered entries, we see that $\pi\in\Grid(\M_{\pi^{\gridded}})^{+p}$. Moreover, there the size of $\M_{\pi^{\gridded}}$ is bounded (the bounds are determined by $t$, $u$, $m$, and the function $f(m)$, and thus in particular they are determined by $\C$), and $\M_{\pi^\gridded}$ itself is of the form specified in the theorem.

To complete the proof we must find a single matrix $\M'$ satisfying the desired conditions such that every permutation $\pi\in\C$ lies in $\Grid(\M')^{+p}$. To do so we set $\M'$ equal to the direct sum of every matrix $\M_{\pi^\gridded}$ which arises in this manner; as these matrices have bounded size there can be only finitely many of them, and thus $\M'$ is itself finite.
\end{proof}

\section{Restricting to a Component}
\label{sec-component}

Theorem~\ref{thm-Gxi-concentration} shows that every class of growth rate less than $\xi$ can be gridded by a concentrated cell class, and thus Theorem~\ref{thm-slicing} implies that all such classes are contained in $\Grid(\M)^{+p}$ for a matrix $\M$ in which the nonmonotone cells are---in a sense to made more formal briefly---isolated. Recall that by Proposition~\ref{prop-extension-gr}, upper and lower growth rates are unaffected by taking $p$-point extensions. Therefore we may ignore those entries and concern ourselves with classes which are $\M$-griddable for a matrix $\M$ satisfying the conclusion of Theorem~\ref{thm-slicing}.

Moreover, Observation~\ref{obs-gridded-gr} shows that the (upper and lower) growth rates of subclasses of such a grid class $\Grid(\M)$ are equal to the equivalent growth rates of the sequences counting $\M$-griddings of their members. Therefore, given a class $\C\subseteq\Grid(\M)$, we denote by $\C^{\gridded}$ the set of all $\M$-griddings of members of $\C$ (the matrix $\M$ will always be clear from context).

Our goal in this section is to show that these growth rates are achieved either by the restriction of the gridded permutations $\C^\gridded$ to a single cell of $\M$ (and thus to a subclass of $\G_\xi$) or are integral. To establish this we must first introduce a graph related to the matrix $\M$.

\begin{figure}
\begin{footnotesize}
\begin{center}
	$\M=\left(\begin{array}{ccccccc}
		\emptygray&\G&\emptygray&\emptygray&\emptygray&\G&\emptygray\\
		\emptygray&\emptygray&\emptygray&\G&\emptygray&\emptygray&\G\\
		\emptygray&\emptygray&\emptygray&\emptygray&\emptygray&\{1\}&\emptygray\\
		\G&\emptygray&\Av(12)&\emptygray&\emptygray&\emptygray&\emptygray\\
		\emptygray&\G&\emptygray&\emptygray&\emptygray&\G&\emptygray\\
		\emptygray&\emptygray&\Av(21)&\emptygray&\G&\emptygray&\emptygray
	\end{array}\right)$
	\quad\quad
	\begin{tikzpicture}[xscale=0.9, yscale=0.41, baseline=(current bounding box.center)]
		\node (31) at (3,1) {$\Av(21)$};
		\node (51) at (5,1) {$\G$};
		\node (22) at (2,2) {$\G$};
		\node (62) at (6,2) {$\G$};
		\node (13) at (1,3) {$\G$};
		\node (33) at (3,3) {$\Av(12)$};
		\node (64) at (6,4) {$\{1\}$};
		\node (45) at (4,5) {$\G$};
		\node (75) at (7,5) {$\G$};
		\node (26) at (2,6) {$\G$};
		\node (66) at (6,6) {$\G$};
		\draw [ultra thick] (13)--(33)--(31)--(51);
		\draw [ultra thick] (22)--(62)--(64)--(66)--(26)--(22);
		\draw [ultra thick] (45)--(75);
	\end{tikzpicture}
\end{center}
\caption{A matrix of permutation classes $\M$ on the left and its cell graph on the right.}
\label{fig-cell-graph}
\end{footnotesize}
\end{figure}
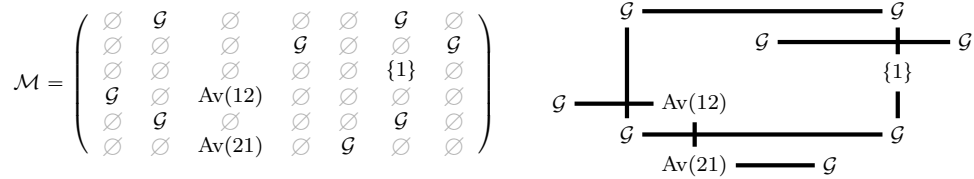

Given a matrix $\M$ of permutation classes, its \emph{cell graph} is the graph on the vertices
\[
	\{(i,j)\st \M_{i,j}\neq\emptyset\}
\]
in which $(i,j)$ and $(k,\ell)$ are adjacent if the corresponding cells share either a row or a column and there are no nonempty cells between them in that row or column. We further label the vertex $(i,j)$ in this graph (if such a vertex exists, as empty cells do not correspond to vertices) by the class $\M_{i,j}$. Figure~\ref{fig-cell-graph} shows an example.

We say that a \emph{connected component} of the matrix $\M$ is a submatrix of $\M$ whose cells correspond to a connected component of the cell graph of $\M$. Given a connected component $\K$ of $\M$ and an $\M$-gridded permutation $\pi^\gridded$, we define the \emph{restriction} of $\pi^\gridded$ to $\K$ as the gridded permutation formed by the entries of $\pi^\gridded$ lying in the cells of $\K$. If $\C\subseteq\Grid(\M)$ is a class, $\C^\gridded$ is the set of all $\M$-griddings of members of $\C$, and $\K$ is a connected component of $\M$, then we say that the \emph{restriction} of $\C$ to $\K$ is the set of all restrictions of the gridded permutations $\pi^\gridded\in\C^\gridded$ to the cells of $\K$.

A given $\M$-gridded permutation $\pi^\gridded\in\C^\gridded$ is uniquely determined by its restrictions to the connected components of $\M$. Thus if $\C\subseteq\Grid(\M)$ then $\C^\gridded$ can be expressed as the Cartesian product of restrictions of $\C$ to connected components of $\M$. The result below follows readily from this observation.

\begin{proposition}
[Vatter~{\cite[Proposition~2.10]{vatter:small-permutati:}}]
\label{prop-grid-component}
Suppose $\C\subseteq\Grid(\M)$. The upper growth rate of $\C$ is the maximum of the upper growth rates of its restrictions to connected components of $\M$.
\end{proposition}

If $\M$ satisfies the conclusion of Theorem~\ref{thm-slicing} then we know that each of its connected components is either a single cell or contains only monotone cells. Therefore by Proposition~\ref{prop-grid-component} and our previous results, the upper growth rates under $\xi$ are all achieved either by subclasses of $\G_\xi$ or by subclasses of monotone grid classes. We consider the growth rates achieved by subclasses of $\G_\xi$ in the next two sections. For the remainder of this section, we show that the only growth rates of subclasses of monotone grid classes under $1+\varphi\approx 2.61803$ are $0$, $1$, and $2$. Note that while Bevan~\cite{bevan:growth-rates-of:} has characterized the growth rates of monotone grid classes themselves (in terms of the spectral radii of graphs associated with their gridding matrices), we cannot appeal to his results here because we must consider \emph{subclasses} of monotone grid classes.

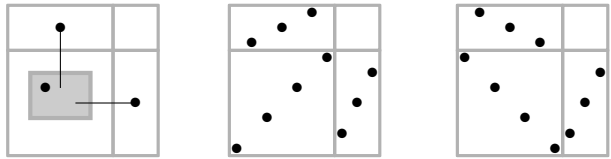
\begin{figure}
\begin{center}
	\begin{tikzpicture}[scale=0.2, baseline=(current bounding box.center)]
		\draw [darkgray, fill=lightgray, ultra thick, line cap=round] (2,3) rectangle (6,6);
		\plotpermbox{1}{1}{7}{7};
		\plotpermbox{1}{8}{7}{10};
		\plotpermbox{8}{1}{10}{7};
		\plotpermbox{8}{8}{10}{10};
		\plotpartialperm{3/5,4/9,9/4};
		\draw (4,9)--(4,5);
		\draw (9,4)--(5,4);
	\end{tikzpicture}
\quad\quad
	\begin{tikzpicture}[scale=0.2, baseline=(current bounding box.center)]
		\plotpermbox{1}{1}{7}{7};
		\plotpermbox{1}{8}{7}{10};
		\plotpermbox{8}{1}{10}{7};
		\plotpermbox{8}{8}{10}{10};
		\plotperm{1,8,3,9,5,10,7,2,4,6};
	\end{tikzpicture}
\quad\quad
	\begin{tikzpicture}[scale=0.2, baseline=(current bounding box.center)]
		\plotperm{7,10,5,9,3,8,1,2,4,6};
		\plotpermbox{1}{1}{7}{7};
		\plotpermbox{1}{8}{7}{10};
		\plotpermbox{8}{1}{10}{7};
		\plotpermbox{8}{8}{10}{10};
	\end{tikzpicture}
\end{center}
\caption{From left to right, a hook-separated rectangle and two examples of unavoidable structures arising from arbitrarily large independent sets of hook-separated rectangles.}
\label{fig-hooksep-rectangle}
\end{figure}

Our technique begins similarly to that used to prove Theorem~\ref{thm-slicing}. As in that proof, our first step is to define a family of rectangles and prove that they have bounded independence number. Let $\pi^\gridded$ denote an arbitrary gridded permutation. A \emph{hook-separated rectangle} is an axis-parallel rectangle $R$ such that
\begin{itemize}
\item $R$ is fully contained in one cell of $\pi^\gridded$,
\item $\pi(R)$ is nonempty,
\item $\pi^\gridded$ has an entry in another cell which separates $R$ horizontally, and
\item $\pi^\gridded$ has an entry in another cell which separates $R$ vertically.
\end{itemize}
An example is shown in Figure~\ref{fig-hooksep-rectangle}. Note that, as in our definition of separated nonmonotone rectangles, we do not insist that the two separating entries actually separate the entries of $\pi(R)$; indeed, we do not insist that $R$ enclose two entries. The following is the analogue of Proposition~\ref{prop-sep-nonmono-rects} in this context.

\begin{proposition}[Essentially Vatter~{\cite[Propositions~A.11 and A.12]{vatter:small-permutati:}}]
\label{prop-hook-sep-rects}
Let $\C\subseteq\Grid(\M)$. If there is no bound on the size of independent sets of hook-separated rectangles in $\M$-gridded members of $\C$ then $\lgr(\C)\ge 1+\varphi\approx 2.61803$.
\end{proposition}
\begin{proof}
Suppose that $\M$ is a $t\times u$ matrix of permutation classes and that $\C\subseteq\Grid(\M)$ contains gridded permutations with at least $4tum^8$ independent hook-separated rectangles for every value of $m$. Fix a value of $m$ and take $\pi^\gridded$ to be such a permutation. At least $4m^8$ of these hook-separated rectangles must lie in the same cell of $\pi^\gridded$. Furthermore, at least $m^8$ of these rectangles are separated in the same pair of directions (above and right, above and left, below and right, or below and left); without loss assume they are separated both above and to the right as in Figure~\ref{fig-hooksep-rectangle}. By the Erd\H{o}s--Szekeres Theorem, we see that some cell of $\pi^\gridded$ contains an increasing or decreasing sequence of at least $m^4$ independent hook-separated rectangles, each separated in the same pair of directions. Applying the Erd\H{o}s--Szekeres Theorem two more times, once for each set of separating entries, we see that some cell of $\pi^\gridded$ contains a monotone sequence of $m$ entries such that each pair of these entries is separated both horizontally and vertically, and that these separating entries themselves form monotone subpermutations. As this holds for arbitrary values of $m$, we may conclude that $\C$ contains the entire monotone grid class of one of the eight $\zpm$ matrices of the form
\[
	\fnmatrix{rr}{\ast&0\\\ast&\ast},
\]
where each $\ast$ denotes either $1$ or $-1$. The proof is completed by noting that each of these monotone grid classes has growth rate $1+\varphi$. This fact follows from either of two works by Bevan~\cite{bevan:growth-rates-of:geom,bevan:growth-rates-of:}, or the simplification of Bevan's work by Albert and Vatter~\cite{albert:an-elementary-p:}, or alternatively, a proof from first principles is given in \cite[Proposition~A.11]{vatter:small-permutati:}.
\end{proof}

We now specialize to subclasses of monotone grid classes, which are called \emph{monotone griddable classes}. The following is the analogue of Theorem~\ref{thm-slicing} in this context.

\begin{proposition}
\label{prop-mono-griddable-vector}
Suppose that $\C\subseteq\Grid(M)$ for a $\zpm$ matrix $M$. If there is a bound on the size of independent sets of hook-separated rectangles in $M$-gridded members of $\C$ then $\C\subseteq\Grid(M')$ for a $\zpm$ matrix $M'$ in which every is either the only component in its row or the only component in its column.
\end{proposition}
\begin{proof}
Suppose $\C\subseteq\Grid(M)$ for a $\zpm$ matrix of size $t\times u$. Let $\pi\in\C$ be arbitrary and fix an $M$-gridding $\pi^\gridded$ of it. Let $\mathfrak{R}_{\pi^\gridded}$ denote the set of all hook-separated rectangles of $\pi^\gridded$. By our hypotheses, there is a constant $m$ not depending on $\pi^\gridded$ such that $\mathfrak{R}_{\pi^\gridded}$ has no independent set of size greater than $m$. Therefore by Theorem~\ref{rectangles-lemma} there is a set $\mathfrak{L}_{\pi^\gridded}$ of at most $f(m)$ vertical and horizontal lines which slice every separated nonmonotone rectangle in $\mathfrak{R}_{\pi^\gridded}$.

The $f(m)$ lines in $\mathfrak{L}_{\pi^\gridded}$ created a refinement (of bounded size) of the gridding denoted by $\pi^\gridded$. We record this refined gridding in a matrix $M_{\pi^\gridded}$. This matrix records a $1$ if the entries of the corresponding cell in the refined gridding of $\pi^\gridded$ are increasing, $-1$ if they are decreasing, and $0$ if the cell is empty.

The matrix $M_{\pi^\gridded}$ clearly satisfies the conclusion of the proposition: otherwise it would have a connected component containing a (nonzero) entry adjacent to other nonzero entries in both its row and its column, but that cell would define a hook-separated rectangle and thus would have been sliced by the lines $\mathfrak{L}_{\pi^\gridded}$.

We complete the proof in the standard manner. There is a bound on the sizes of the matrices $M_{\pi^\gridded}$ which can arise in the procedure outlined above. By taking $M'$ to be the direct sum of every matrix $M_{\pi^\gridded}$ which arises in this manner, we see that $M'$ also satisfies the conclusion of the proposition and every permutation $\pi\in\C$ is $M'$-griddable.
\end{proof}

At the end of the previous section we established that if the upper growth rate of $\C$ is less than $\xi$ then its upper/lower growth rate(s) are given either by a subclass of $\G_\xi$ or by a subclass of a monotone grid class. With Proposition~\ref{prop-mono-griddable-vector}, we see that in the latter case we can assume that the growth rates are given by a subclass of the monotone grid class of a (row or column) vector. In the remainder of this section we establish that subclasses of monotone grid classes of vectors necessarily have integral growth rates.

The consideration of monotone grid classes of vectors dates back to the work of Atkinson, Murphy, and Ru\v{s}kuc~\cite{atkinson:partially-well-:} and Albert, Atkinson, and Ru\v{s}kuc~\cite{albert:regular-closed-:}, who called them ``W-classes'' because the plot of a typical member of the class $\Grid({-1}\ 1\ {-1}\ 1)$ resembles the letter W. There is a natural encoding of members of monotone grid classes of vectors (later generalized to members of arbitrary geometric grid classes in \cite{albert:geometric-grid-:}).

Suppose that $M$ is a $\zpm$ matrix of size $t\times 1$ (meaning in our notation that it is a row vector of length $t$) and that $\C\subseteq\Grid(M)$. As we are interested only in growth rates, we consider the set of all $M$-griddings of members of $\C$, denoted $\C^\gridded$, though it is possible to use this encoding to determine the exact enumeration of such classes (see \cite{albert:regular-closed-:} for the vector version and \cite{albert:geometric-grid-:} for the more general geometric version). First, label the cells of $M$ from $1$ to $t$ from left-to-right (for concreteness only). Given a gridded permutation $\pi^\gridded\in\C^\gridded$ of length $n$, its encoding is a word of length $n$ over the alphabet $\{1,2,\dots,t\}$. The $i$th letter of this encoding denotes the cell in which the entry of value $i$ in $\pi^\gridded$ lies. Conversely, each word over the alphabet $\{1,2,\dots,t\}$ describes precisely one permutation of $\Grid(M)$ in this manner.

If $\pi\in\C$ has an $M$-gridding $\pi^\gridded$ encoded by the word $w$, then every subword of $w$ encodes an $M$-gridding of a subpermutation of $\pi$. Therefore the set of all encodings of the gridded permutations $\C^\gridded$ forms a \emph{subword-closed language} (other authors, notably Simon~\cite{simon:piecewise-testa:}, who was one of the first to study these languages, use the term ``piecewise testable'').

Recall that the collection of regular languages over a finite alphabet $\Sigma$ is defined recursively as follows. First, the empty set, the singleton $\{\emptyword\}$ containing only the empty word, and the singletons $\{a\}$ for each $a\in\Sigma$ are all regular languages. Then, given two regular languages $K,L\subseteq\Sigma^\ast$, their union $K\cup L$, their concatenation $KL=\{vw\st v\in K\mbox{ and }w\in L\}$, and the star $K^\ast=\{v^{(1)}\cdots v^{(m)}\st v^{(1)},\dots,v^{(m)}\in K\}$ are also regular. Another basic fact we need is that the complement of a regular language is itself regular; we refer the unfamiliar reader to Bousquet-M\'elou~\cite[Section 2]{bousquet-melou:rational-and-al:} for a combinatorial introduction to regular languages.

It follows quickly from Higman's Lemma that every subword-closed language is regular. If $\LL$ is subword-closed then there are only finitely many minimal (in the subword ordering) words not in $\LL$, say $B=\{b_1,\dots,b_\ell\}$. Therefore the complement of $\LL$ is the union of the languages of words containing $b_i$ for all $i$. Those languages are easily shown to be regular, so the complement of $\LL$, and hence $\LL$ itself are also regular.

In the following result we use this fact to show that growth rates of subclasses of monotone grid classes of $\zpm$ vectors are necessarily integral. In this result we have extended our notation for permutation classes to languages in the obvious manner; thus given a language $\LL$, $\LL_n$ denotes its members of length $n$ and we call the limit of $\sqrt[n]{|\LL_n|}$ as $n\to\infty$, if it exists, the growth rate of $\LL$. The proof below is due to Michael Albert [private communication].

\begin{proposition}
\label{prop-subword-integer}
The growth rate of every subword-closed language exists and is integral.
\end{proposition}
\begin{proof}
We claim that every subword-closed language $\LL\subseteq\Sigma^\ast$ can be expressed as a finite union of regular expressions of the form $\ell_1\Sigma_1^\ast\cdots\ell_k\Sigma_k^\ast\ell_{k+1}$ for letters $\ell_i\in\Sigma$ and subsets $\Sigma_i\subseteq\Sigma$. This follows by induction on the regular expression defining $\LL$. The base cases where $\LL$ is empty, $\{\emptyword\}$, or a single letter are trivial. If the regular expression defining $\LL$ is a union or a concatenation then the claim follows inductively. The only other case is when this regular expression is a star, $\LL=E^\ast$. In this case, because $\LL$ is subword-closed, we see that $\LL=\Lambda^\ast$ where $\Lambda\subseteq\Sigma$ is the set of letters occurring in $E$.

With this claim established, it follows that $\lim\sqrt[n]{|\LL_n|}$ is equal to the size of the largest set $\Sigma_i$ occurring in such an expression for $\LL$.
\end{proof}

\begin{corollary}
\label{cor-mono-grid-grs}
The only growth rates of monotone griddable classes less than $1+\varphi\approx 2.61803$ are $0$, $1$, and $2$.
\end{corollary}

We collect what we have proved in this section below. The following is the only result of the past two sections used in what follows.

\begin{corollary}
\label{cor-slicing-gr}
Suppose that the permutation class $\C$ is $\G$-griddable for a concentrated class $\G$ and that $\lgr(\C)<1+\sqrt{2}\approx 2.41421$. Then either $\gr(\C)$ exists and is integral or $\ugr(\C)$ is equal to the upper growth rate of a subclass of $\G$.
\end{corollary}

\section{Well-Quasi-Order}
\label{sec-wqo}

One of the final ingredients needed for the proof of our main result is to show that the grid classes $\G_\gamma$ for $\gamma<\xi$ are well-quasi-ordered. In the characterization of growth rates below $\kappa$, well-quasi-order is an immediate consequence of the fact that $\G_\gamma$ contains only finitely many simple permutations for $\gamma<\kappa$ (Proposition~\ref{prop-Ggamma-simples}) and Theorem~\ref{thm-fin-simples-wqo}, which shows that classes with only finitely many simple permutations are necessarily well-quasi-ordered. For $\kappa\le\gamma<\xi$, however, a different argument is required. Because the infinite antichain $U^o$ is contained in $\G_\xi$ (a consequence of Proposition~\ref{prop-U-sum-indecomp-kids}), $\G_\xi$ is not well-quasi-ordered and so the result obtained here is best possible in this sense.

\begin{figure}
\begin{footnotesize}
\begin{center}
	\begin{tikzpicture}[scale={(15*0.25)/3}]
		\draw [thick, line cap=round] (0,2)--(1,1);
		\draw [thick, line cap=round] (0,0)--(2,2);
		\draw [thick, line cap=round] (2,1)--(3,0);
		\draw [darkgray, ultra thick, line cap=round] (0,0)--(3,0);
		\draw [darkgray, ultra thick, line cap=round] (0,1)--(3,1);
		\draw [darkgray, ultra thick, line cap=round] (0,2)--(3,2);
		\draw [darkgray, ultra thick, line cap=round] (0,0)--(0,2);
		\draw [darkgray, ultra thick, line cap=round] (1,0)--(1,2);
		\draw [darkgray, ultra thick, line cap=round] (2,0)--(2,2);
		\draw [darkgray, ultra thick, line cap=round] (3,0)--(3,2);
		\plotpartialperm{0.2/1.8, 0.4/0.4, 0.5/0.5, 0.8/1.2, 1.1/1.1, 1.9/1.9, 2.3/0.7};
	\end{tikzpicture} 
\end{center}
\end{footnotesize}
\caption[The permutation $6125473$ lies in a geometric grid class.]{The permutation $6125473$ lies in the geometric grid class of $\fnmatrix{rrr}{-1&1&0\\1&0&-1}$.}
\label{fig-ggc-example}
\end{figure}
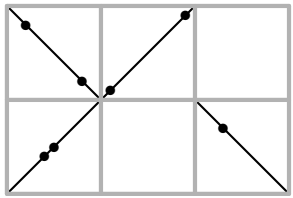

First we briefly discuss geometric grid classes. The \emph{standard figure}, denoted $\Lambda_M$, of a $\zpm$ matrix $M$ is the figure in $\mathbb{R}^2$ consisting of two types of line segments for every pair of indices $i,j$ such that $M(i,j)\neq 0$:
\begin{itemize}
\item the increasing open line segment from $(i-1,j-1)$ to $(i,j)$ if $M(i,j)=1$ or
\item the decreasing open line segment from $(i-1,j)$ to $(i,j-1)$ if $M(i,j)=-1$.
\end{itemize}
The \emph{geometric grid class} of $M$, denoted $\Geom(M)$, is the set of all permutations that can be drawn on $\Lambda_M$ (in the sense defined in the introduction and the same as in Proposition~\ref{prop-321-lines}---an example is shown in Figure~\ref{fig-ggc-example}).

Geometric grid classes were first introduced by Albert, Atkinson, Bouvel, Ru\v{s}kuc, and Vatter~\cite{albert:geometric-grid-:}, who showed that for every $\zpm$ matrix $M$, there is a finite alphabet $\Sigma$ (consisting of one letter per nonzero cell of $M$, called the \emph{cell alphabet}) and a mapping $\bij\st\Sigma^\ast\to\Geom(M)$ such that whenever $v$ is contained in $w$ as a subword then the corresponding permutations satisfy $\bij(v)\le\bij(w)$. Indeed, this encoding is a generalization of the one used for grid classes of row vectors (which are always geometric grid classes) in the previous section. Given this encoding, it follows immediately from Higman's Lemma that all geometric grid classes are well-quasi-ordered.

We say that the class $\C$ is \emph{geometrically griddable} if it is contained in some geometric grid class, i.e., if $\C\subseteq\Geom(M)$ for any finite $\zpm$ matrix $M$.  In a later paper, Albert, Ru\v{s}kuc, and Vatter~\cite{albert:inflations-of-g:} studied inflations and substitution closures of geometrically griddable classes, and it is these results we appeal to here. The two well-quasi-order results proved in \cite{albert:inflations-of-g:} are the following.
\begin{itemize}
\item If the class $\C$ is a geometrically griddable class and the class $\D$ is well-quasi-ordered then $\C[\D]$ is also well-quasi-ordered~\cite[Proposition 4.1]{albert:inflations-of-g:}.
\item If the class $\C$ is geometrically griddable then $\langle\C\rangle$ is also well-quasi-ordered~\cite[Theorem 4.4]{albert:inflations-of-g:}.
\end{itemize}
(Recall from Section~\ref{sec-background} that $\C[\D]$ is the inflation of $\C$ by $\D$ while $\langle\C\rangle$ denotes the substitution closure of $\C$.) The first of the results above is also an immediate consequence of Higman's Lemma, but the second requires slightly more sophisticated tools. In fact, while not noted in \cite{albert:inflations-of-g:}, a minor adaptation of the proof of this second result yields the following common generalization of both.

\begin{theorem}[An adaptation of the proof of {\cite[Theorem 4.4]{albert:inflations-of-g:}}]
\label{thm-subst-geom-inflate-wqo}
If the class $\C$ is geometrically griddable and the set $X$ of permutations is well-quasi-ordered then the set $\langle\C\rangle[X]$ is also well-quasi-ordered.
\end{theorem}

Note that every finite class is geometrically griddable, so Theorem~\ref{thm-subst-geom-inflate-wqo} includes as a special case Theorem~\ref{thm-fin-simples-wqo}: any class with only finitely many simple permutations is contained in $\langle\C\rangle$ for a finite, and hence geometrically griddable, class $\C$. The simplest nontrivial class of this form is the class of \emph{separable permutations}, $\langle\{1,12,21\}\rangle$, which can also be described as the class of all permutations that can be obtained from the permutation $1$ by repeated sums and skew sums. A result of Stankova~\cite{stankova:forbidden-subse:} which we use below shows that the separable permutations may also be described as $\Av(2413,3142)$\footnote{The graph-theoretic analogues of separable permutations are the $P_4$-free graphs, also known as the co-graphs, which is why the basis of the separable permutations consists of the two permutations whose inversion graphs are isomorphic to $P_4$.}.

Theorem~\ref{thm-subst-geom-inflate-wqo} says much more than this, though. Let $\C$ be any permutation class and define
\[
	I_\C=\{\alpha\in\C\st \text{$\alpha$ is both sum and skew indecomposable}\}
\]
to be the set of indecomposable members of $\C$. It follows inductively that the class $\C$ is a subset of the set $\langle\{1,12,21\}\rangle[I_\C]$: every member of $\C$ either lies in $I_\C$ or is a sum or skew sum of shorter members of $\C$. Thus if the set $I_\C$ is well-quasi-ordered then $\langle\{1,12,21\}\rangle[I_\C]$ is well-quasi-ordered, and because $\C\subseteq\langle\{1,12,21\}\rangle[I_\C]$, we may conclude that $\C$ is itself well-quasi-ordered. We record this implication below.

\begin{corollary}[Cf.~Atkinson, Murphy, and Ru\v{s}kuc~{\cite[Theorem 2.5]{atkinson:partially-well-:}}]
\label{cor-sum-skew-wqo}
The class $\C$ is well-quasi-ordered if and only if its members which are both sum and skew indecomposable are well-quasi-ordered.	
\end{corollary}

For the rest of this section consider $\gamma<\xi$ to be fixed and suppose to the contrary that there is an infinite antichain $A\subseteq\G_\gamma$. By the result above, we may assume that every member of $A$ is both sum and skew indecomposable. Each member $\alpha$ of $A$ lies in $\G_\gamma$ because the growth rate of either $\bigoplus\Sub(\alpha)$ or $\bigominus\Sub(\alpha)$ is less than $\gamma$. Thus we may assume by symmetry and without loss of generality that $\gr(\bigoplus\Sub(\alpha))<\gamma$ for every $\alpha\in A$. It is this fact that we seek to contradict, by showing that $\alpha$ must contain enough sum indecomposable subpermutations that the growth rate of $\bigoplus\Sub(\alpha)$ is greater than $\xi$. Specifically, we make use of the following computation.

\begin{proposition}
\label{prop-wqo-computation-xi}
If the sum indecomposable permutation $\alpha$ contains $3$ sum indecomposable permutations of each of the lengths $3$ and $4$ and itself has length at least $10$ then $\gr(\bigoplus\Sub(\alpha))>\xi$.
\end{proposition}
\begin{proof}
Just as every connected graph contains a connected induced subgraph on one fewer vertex, every sum indecomposable permutation contains a sum indecomposable subpermutation with one fewer entry. Therefore by the hypotheses and Proposition~\ref{prop-enum-oplus-closure} it follows immediately that $\gr(\bigoplus\Sub(\alpha))$ is at least the growth rate of
\[
	\frac{1}{1-\left(x+x^2+3x^3+3x^4+x^5+x^6+x^7+x^8+x^9+x^{10}\right)},
\]
which is greater than $2.30528$, more than $0.00006$ greater than $\xi$.
\end{proof}

The first step of our argument is to show that infinitely many members of $A$ must contain $321$. To do so we must first introduce a few facts about $321$-avoiding antichains. Following Huczynska and Ru\v{s}kuc~\cite{huczynska:well-quasi-orde:}, we say that a \emph{double-ended fork} is the graph formed from a path by adding four vertices of degree one, two adjacent to one end of the path and two adjacent to the other, as shown below.
\begin{center}
	\begin{tikzpicture}[scale=0.4, yscale=0.5]
		\draw (1,0)--(7,0);
		\draw (0,1)--(1,0)--(0,-1);
		\draw (8,1)--(7,0)--(8,-1);
		\plotpartialperm{0/-1,0/1,1/0,2/0,3/0,4/0,5/0,6/0,7/0,8/-1,8/1};
	\end{tikzpicture}
\end{center}
It is clear that the set of double-ended forks forms an antichain of graphs in the induced subgraph ordering.

Let $U$ denote the set of all permutations $\pi$ for which $G_\pi$ is isomorphic to a double-ended fork on at least $6$ vertices. By direct construction it can be seen that the members of $U$ fall into four slightly different types, as indicated in Figure~\ref{fig-antichains}, and it follows that $U$ is an infinite antichain of $321$-avoiding permutations. We label the sets consisting of the four types of members of $U$ by $U^o$ (for odd length), $U^e$ (for even length), $(U_o)^{-1}$, and $(U^e)^{-1}$. The result we quote below shows that $U$ is in some sense the minimal antichain in $\Av(321)$. Its proof uses a generalization of the well-quasi-order properties of geometric grid classes.

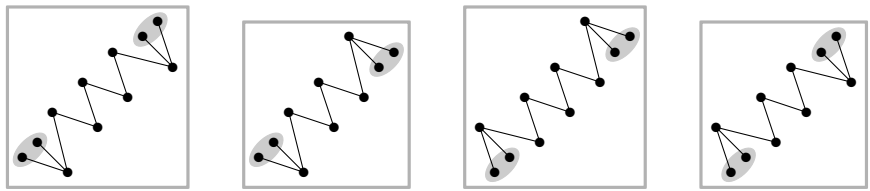
\begin{figure}
\begin{footnotesize}
\begin{center}
	\begin{tikzpicture}[scale=0.2]
		% Intervals:
		\draw[lightgray, fill, rotate around={-45:(1.5,2.5)}] (1.5,2.5) ellipse (20pt and 40pt);
		\draw[lightgray, fill, rotate around={-45:(9.5,10.5)}] (9.5,10.5) ellipse (20pt and 40pt);
		% The permutation:
		\plotpermbox{0.5}{0.5}{11.5}{11.5};
		\plotpermgraph{2,3,5,1,7,4,9,6,10,11,8};
	\end{tikzpicture}
\quad\quad
	\begin{tikzpicture}[scale=0.2]
		% Intervals:
		\draw[lightgray, fill, rotate around={-45:(1.5,2.5)}] (1.5,2.5) ellipse (20pt and 40pt);
		\draw[lightgray, fill, rotate around={-45:(9.5,8.5)}] (9.5,8.5) ellipse (20pt and 40pt);
		% The permutation:
		\plotpermbox{0.5}{0.5}{10.5}{10.5};
		\plotpermgraph{2,3,5,1,7,4,10,6,8,9};
	\end{tikzpicture}
\quad\quad
	\begin{tikzpicture}[scale=0.2]
		% Intervals:
		\draw[lightgray, fill, rotate around={-45:(2.5,1.5)}] (2.5,1.5) ellipse (20pt and 40pt);
		\draw[lightgray, fill, rotate around={-45:(10.5,9.5)}] (10.5,9.5) ellipse (20pt and 40pt);
		% The permutation:
		\plotpermbox{0.5}{0.5}{11.5}{11.5};
		\plotpermgraph{4,1,2,6,3,8,5,11,7,9,10};
	\end{tikzpicture}
\quad\quad
	\begin{tikzpicture}[scale=0.2]
		% Intervals:
		\draw[lightgray, fill, rotate around={-45:(2.5,1.5)}] (2.5,1.5) ellipse (20pt and 40pt);
		\draw[lightgray, fill, rotate around={-45:(8.5,9.5)}] (8.5,9.5) ellipse (20pt and 40pt);
		% The permutation:
		\plotpermbox{0.5}{0.5}{10.5}{10.5};
		\plotpermgraph{4,1,2,6,3,8,5,9,10,7};
	\end{tikzpicture}
\end{center}
\end{footnotesize}
\caption{The four different types of members of $U$. From left to right, we denote these antichains by $U^o$, $U^e$, $(U^o)^{-1}$, and $(U^e)^{-1}$.}
\label{fig-antichains}
\end{figure}

\begin{theorem}[Albert, Brignall, Ru\v{s}kuc, and Vatter~{\cite[Theorem 9.3]{albert:rationality-for:}}]
\label{thm-wqo-fin-int}
A subclass $\C \subseteq \Av(321)$ is well-quasi-ordered if and only if $\C \cap U$ is finite.
\end{theorem}

Next we argue that $U\not\subseteq\G_\gamma$ for $\gamma<\xi$. Label the members of $U^o$ by length as $\mu_7$, $\mu_9$, $\dots$ and the members of $U^e$ by length as $\mu_6$, $\mu_8$, $\dots$. It is routine to verify the following computation.

\begin{proposition}
\label{prop-U-sum-indecomp-kids}
The sequence of sum indecomposable permutations contained in $\mu_{2i+7}$ for $i\ge 0$ is $1,1,2,3,4^{2i},3,2,1$ while the sequence of sum indecomposable permutations contained in $\mu_{2i+6}$ for $i\ge 0$ is $1,1,2,4^{2i},3,2,1$.
\end{proposition}

There are several consequences of Proposition~\ref{prop-U-sum-indecomp-kids}. First, it immediately implies that the cell classes $\G_\gamma$ for $\gamma<\xi$ do not contain any members of $U^e$ or $(U^e)^{-1}$ because the growth rate of $\bigoplus\Sub(\mu_{2i+6})$ is greater than $\xi$. It further shows that the growth rate of $\bigoplus\Sub(\mu_{2i+7})$ becomes arbitrarily close to $\xi$ as $i\to\infty$ because $\xi$ is the growth rate of a sum closed class whose sum indecomposable members have the enumeration $1,1,2,3,4^\infty$. Therefore $\G_\gamma\cap U^o$ is finite if and only if $\gamma<\xi$. By the first observation of this paragraph, this implies that $\G_\gamma\cap U$ is finite for $\gamma<\xi$, and thus by Theorem~\ref{thm-wqo-fin-int},  $\G_\gamma\cap\Av(321)$ is well-quasi-ordered. Returning to the infinite antichain $A$ under consideration, we see that it can contain only finitely many $321$-avoiding members and thus without loss of generality we may assume that it contains none, i.e., that every member of $A$ contains $321$.

Next, clearly $A$ cannot contain infinitely many separable permutations, as the class of separable permutations is well-quasi-ordered. Therefore we may assume that $A$ contains no separable permutations. Under this assumption, every member of $A$ contains either $2413$ or $3142$. From this and our previous observation about $321$ we may conclude that every member of $A$ contains the sum indecomposable subpermutations $1$, $21$, $231$, $312$, $321$, and either $2413$ or $3142$. To appeal to Proposition~\ref{prop-wqo-computation-xi} we need to find two more sum indecomposable permutations of length $4$ contained in $\alpha$.

First we consider permutations $\alpha\in A$ which contain both $2413$ and $3142$. The inversion graphs of $2413$ and $3142$ are both paths, but we know that the inversion graph of $\alpha$ is not a path because $\alpha$ contains $321$ (so its inversion graph contains a triangle). Therefore the inversion graph of $\alpha$ must contain a connected induced subgraph on $4$ vertices which is not a path, and thus the corresponding entries of $\alpha$ form a third sum indecomposable subpermutation of length $4$. Proposition~\ref{prop-wqo-computation-xi} then shows that the growth rate of $\bigoplus\Sub(\alpha)$ is greater than $\xi$, a contradiction. Therefore we may assume that only finitely many, and hence without loss of generality, none, of the members of $A$ contain both $2413$ and $3142$.

We are reduced to the case where infinitely many members of $A$ contain precisely one of $2413$ or $3142$. Without loss of generality, and by a final appeal to symmetry, let us suppose that all members of $A$ contain $2413$ and avoid $3142$. We need to find two more sum indecomposable permutations of length $4$ that these permutations contain.

We claim that infinitely many members of $A$ must contain $4213$. The result below, which implies this claim, has been observed before, but we include its short proof for completeness.

\begin{proposition}[Albert, Atkinson, and Vatter~{\cite[Section 2]{albert:inflations-of-g:2x4:}}]
\label{prop-structure-av-3142-4213}
The class $\Av(3142, 4213)$ is contained in $\langle\Geom(1\ 1)\rangle$ and hence is well-quasi-ordered.	
\end{proposition}
\begin{proof}
It suffices to show that the simple permutations of $\Av(3142, 4213)$ are contained in the geometric grid class of $(1\ 1)$, i.e., that the simple permutations in this class are all parallel alternations oriented as $//$. First we verify that all the simple permutations in $\Av(3142, 4213)$ are parallel alternations of some orientation. By Schmerl and Trotter's Theorem~\ref{thm-schmerl-trotter}, were $\Av(3142, 4213)$ to contain a simple permutation which is not a parallel alternation, it would contain a simple permutation of length $5$. However, there are precisely $6$ simple permutations of length $5$---$24153$, $25314$, $31524$, $35142$, $41352$, and $42513$---and it is easily checked that each of these contains either $3142$ or $4213$. Now that we know the simple permutations in $\Av(3142, 4213)$ are all parallel alternations, it can also easily be checked that those of length $6$ or longer must be oriented as $//$, completing the proof.
\end{proof}

Immediately from Proposition~\ref{prop-structure-av-3142-4213} we see that infinitely many members of $A$ must contain $4213$. Moreover, the class $\Av(3142,2431)$ is a symmetry of $\Av(3142, 4213)$, so infinitely many members of $A$ must also contain $2431$. This shows that infinitely many members of $A$ contain $3$ sum indecomposable subpermutations of length $4$, allowing us to appeal to Proposition~\ref{prop-wqo-computation-xi} a final time to obtain a contradiction. We have therefore established the goal of this section.

\begin{theorem}
\label{thm-G-gamma-wqo}
The cell class $\G_\gamma$ is well-quasi-ordered if and only if $\gamma<\xi$.
\end{theorem}

\section{The Phase Transition to Uncountably Many Growth Rates}
\label{sec-phase-transition}

We can now establish the first part of Theorem~\ref{thm-xi-main}, namely that there are only countably many upper growth rates of permutation classes below $\xi$ but uncountably many upper growth rates in every open neighborhood of it. Recall that the uncountability part of this result was established in \cite{vatter:permutation-cla:lambda:}, so we need only show the countability result.

Consider a class $\C$ with $\ugr(\C)<\xi$. By Proposition~\ref{prop-G-gamma-grids}, there is some positive integer $k$ such that $\C$ is $\G_{(\xi-1/k)}$-griddable. By Theorem~\ref{thm-Gxi-concentration}, $\G_{(\xi-1/k)}$ is concentrated, because it is a subclass of $\G_\xi$. Therefore, Corollary~\ref{cor-slicing-gr} shows that either $\gr(\C)$ is an integer or it equals the upper growth rate of a subclass of $\G_{(\xi-1/k)}$. Furthermore, each class $\G_{(\xi-1/k)}$ is well-quasi-ordered (Theorem~\ref{thm-G-gamma-wqo}) and therefore contains only countably many subclasses (Proposition~\ref{prop-wqo-subclasses-countable}).

Thus despite the fact that there are uncountably many permutation classes of growth rate $\kappa\approx 2.20557$, the \emph{growth rates} of classes under $\xi\approx 2.30522$ can be expressed as the countable union of countable sets. Thus there are only countably many such growth rates, establishing the first part of Theorem~\ref{thm-xi-main}.

It remains to prove the second part of Theorem~\ref{thm-xi-main}, that all upper growth rates under $\xi$ are proper growth rates and are achieved by sum closed classes. In this proof we make extensive use of two types of permutation classes. First, we call the class $\C$ \emph{grid irreducible} if it is not $\G$-griddable for any proper subclass $\G\subsetneq\C$. The characterization of grid irreducible classes is essentially trivial, given the results established so far.

\begin{proposition}[Vatter~{\cite[Proposition~3.3]{vatter:small-permutati:}}]
\label{prop-grid-irreduce}
The class $\C$ is grid irreducible if and only if $\C=\{1\}$ or if for every $\pi\in\C$ either $\bigoplus\Sub(\pi)\subseteq\C$ or $\bigominus\Sub(\pi)\subseteq\C$.
\end{proposition}
\begin{proof}
It follows immediately that the only finite grid irreducible class is $\{1\}$.

Now suppose that $\C$ is an infinite permutation class. If there is some $\pi\in\C$ such that neither $\bigoplus\Sub(\pi)$ nor $\bigominus\Sub(\pi)$ is contained in $\C$ then $\C$ is $\left(\C\cap\Av(\pi)\right)$-griddable by Theorem~\ref{thm-gridding-characterization}, and thus $\C$ is not grid irreducible. In the other direction, if $\C$ contains arbitrarily long sums or skew sums of each of its elements then Theorem~\ref{thm-gridding-characterization} shows that it cannot be gridded by a proper subclass.
\end{proof}

We present an example in the conclusion showing that there are classes (containing infinite antichains) which are not $\G$-griddable for any grid irreducible class $\G$. However, in the case of well-quasi ordered classes we can be much more explicit than Proposition~\ref{prop-grid-irreduce}, as we show below.

\begin{proposition}
\label{prop-grid-irreduce-wqo}
If the class $\C$ is well-quasi-ordered then it is $\G$-griddable for the grid irreducible class $\G=\{\pi\st \mbox{either } \bigoplus\Sub(\pi)\subseteq\C \mbox{ or } \bigominus\Sub(\pi)\subseteq\C\}$.
\end{proposition}
\begin{proof}
As shown in the proof of Proposition~\ref{prop-grid-irreduce}, if neither $\bigoplus\Sub(\pi_1)$ nor $\bigominus\Sub(\pi_1)$ is contained in $\C$ then it is $\left(\C\cap\Av(\pi_1)\right)$-griddable. We may then apply this to $\C\cap\Av(\pi_1)$ to see that if neither $\bigoplus\Sub(\pi_2)$ nor $\bigominus\Sub(\pi_2)$ is contained in $\C\cap\Av(\pi_1)$ for some $\pi_2\in\C\cap\Av(\pi_1)$ then $\C\cap\Av(\pi_1)$ is $\left(\C\cap\Av(\pi_1,\pi_2)\right)$-griddable. Continuing in this manner, we can construct a descending chain of classes
\[
	\begin{array}{ccccccc}
	\C
	&\supsetneq&
	\C\cap\Av(\pi_1)
	&\supsetneq&
	\C\cap\Av(\pi_1,\pi_2)
	&\supsetneq&
	\cdots
	\\
	\verteq&&\verteq&&\verteq
	\\
	\G^{(0)}
	&\supsetneq&
	\G^{(1)}
	&\supsetneq&
	\G^{(2)}
	&\supsetneq&
	\cdots
	\end{array}
\]
such that $\C$ is $\G^{(i)}$-griddable for all $i$. Because $\C$ is well-quasi-ordered, it satisfies the descending chain condition (Proposition~\ref{prop-wqo-subclasses-dcc}) and thus this process must stop at some $\G^{(k)}$. The class $\G^{(k)}$ therefore has the property that either $\bigoplus\Sub(\pi)$ or $\bigominus\Sub(\pi)$ is contained in $\G^{(k)}$ for every $\pi\in\G^{(k)}$, as desired.
\end{proof}

The second type of classes we use are atomic classes. We call the permutation class $\C$ \emph{atomic} if it cannot be expressed as the union of two of its proper subclasses. Below we include a short proof that atomicity is equivalent to the \emph{joint embedding property}: for all $\sigma,\tau\in\C$, there is a $\pi\in\C$ containing both $\sigma$ and $\tau$.

\begin{proposition}
The class $\C$ is atomic if and only if it satisfies the joint embedding property.\footnote{Fra{\"{\i}}ss{\'e}~\cite{fraisse:sur-lextension-:} (see also Hodges~{\cite[Section 7.1]{hodges:model-theory:}}), working in the more general context of relational structures, established another necessary and sufficient condition for atomicity: a class is atomic if and only if it is an age. In the permutation case, an \emph{age} is a class of the form $\Sub(f\st A\to B)$ for a bijection $f$ between linearly ordered sets $A$ and $B$.}	
\end{proposition}
\begin{proof}
Suppose that $\C=\D\cup\E$ for proper subclasses $\D,\E\subsetneq\C$. We must have $\D\neq\E$, so there are permutations $\sigma\in\D\setminus\E$ and $\tau\in\E\setminus\D$. Clearly no permutation in $\C$ can contain both $\sigma$ and $\tau$, as then it would lie in neither $\D$ nor $\E$. In the other direction, if there are permutations $\sigma,\tau\in\C$ such that no permutation in $\C$ contains both then $\C$ is the union of its proper subclasses $\C\cap\Av(\sigma)$ and $\C\cap\Av(\tau)$.
\end{proof}

Our interest in atomicity lies in the following (folklore) result, and even more so in the result we derive from it.

\begin{proposition}
\label{prop-wqo-atomic-union}
Every well-quasi-ordered permutation class can be expressed as a finite union of atomic classes.
\end{proposition}
\begin{proof}
Consider the binary tree whose root is the well-quasi-ordered class $\C$ and in which the children of a non-atomic class $\D$ are two proper subclasses $\D^{(1)},\D^{(2)}\subsetneq\D$ such that $\D=\D^{(1)}\cup\D^{(2)}$. Because $\C$ is well-quasi-ordered its subclasses satisfy the descending chain condition (Proposition~\ref{prop-wqo-subclasses-dcc}), so this tree contains no infinite paths and thus is finite by K\"onig's Lemma; its leaves are the desired atomic classes.
\end{proof}

The following result follows easily from Proposition~\ref{prop-wqo-atomic-union}; for details we refer to \cite[Proposition~1.20]{vatter:small-permutati:}.

\begin{proposition}
\label{prop-wqo-atomic-gr}
The upper growth rate of a well-quasi-ordered permutation class is equal to the greatest upper growth rate of its atomic subclasses.
\end{proposition}

We need a final preparatory result.

\begin{proposition}[Vatter~{\cite[Proposition~3.5]{vatter:small-permutati:}}]
\label{prop-atomic-grid-irreduce}
An atomic class is grid irreducible if and only if it is sum or skew closed.
\end{proposition}
\begin{proof}
Every sum or skew closed class is grid irreducible (by Theorem~\ref{thm-gridding-characterization} or Proposition~\ref{prop-grid-irreduce}) so it suffices to prove the reverse direction. Suppose that $\C$ is atomic but neither sum nor skew closed. Thus there are permutations $\sigma,\tau\in\C$ such that $\bigoplus\Sub(\sigma),\bigominus\Sub(\tau)\not\subseteq\C$. Because $\C$ is atomic, we can find a permutation $\pi\in\C$ containing both $\sigma$ and $\tau$. However, $\C$ cannot contain arbitrarily long sums or skew sums of $\pi$, so it is $\left(\C\cap\Av(\pi)\right)$-griddable (again by Theorem~\ref{thm-gridding-characterization}). This shows that $\C$ is not grid irreducible, completing the proof.
\end{proof}

We can now prove our main result.

\begin{theorem}
\label{thm-xi-main}
There are only countably many growth rates of permutation classes below $\xi$ but uncountably many growth rates in every open neighborhood of it. Moreover, every growth rate of permutation classes below $\xi$ is achieved by a sum closed permutation class.
\end{theorem}
\begin{proof}
We have already established the first part of the theorem, so let $\C$ be a permutation class with $\ugr(\C)<\xi$. Choose some $\gamma$ between $\ugr(\C)$ and $\xi$ and let $\G^{(1)}=\C\cap\G_\gamma$. By Proposition~\ref{prop-G-gamma-grids}, $\C$ is $\G_{\gamma}$-griddable and thus also $\G^{(1)}$-griddable. Furthermore, $\G^{(1)}$ is concentrated by Theorem~\ref{thm-Gxi-concentration}. Therefore, by Corollary~\ref{cor-slicing-gr}, either the growth rate of $\C$ exists and is integral, in which case we are done, or $\ugr(\C)$ is equal to the upper growth rate of a subclass of $\G^{(1)}$. In the latter case, denote this subclass by $\C^{(1)}$. Note that by Theorem~\ref{thm-G-gamma-wqo}, $\G_\gamma$ is well-quasi-ordered, and hence so are $\G^{(1)}$ and $\C^{(1)}$.

We now continue this process. Suppose we have found that $\ugr(\C)=\ugr(C^{(i)})$ for a well-quasi-ordered subclass $\C^{(i)}\subseteq\C$. By Proposition~\ref{prop-wqo-atomic-gr}, $\ugr(\C^{(i)})$ is equal to the greatest upper growth rate of its atomic subclasses; let $\A^{(i)}$ denote this atomic subclass. Because $\A^{(i)}$ is also well-quasi-ordered, Proposition~\ref{prop-grid-irreduce-wqo} shows that it is $\G^{(i+1)}$-griddable for a grid irreducible subclass $\G^{(i+1)}\subseteq\A^{(i)}$. As $\G^{(i+1)}\subseteq\G_\gamma$, it is concentrated, and so by Corollary~\ref{cor-slicing-gr}, either $\gr(\C)$ exists and is integral or $\ugr(C)=\ugr(\A^{(i)})=\ugr(C^{(i+1)})$ for a subclasses $\C^{(i+1)}\subseteq\G^{(i+1)}$. Repeating this process indefinitely, we either find that the growth rate of $\C$ exists and is integral or we build an infinite descending chain of subclasses of $\C$,
\[
	\C
	\supseteq
	\G^{(1)}
	\supseteq
	\C^{(1)}
	\supseteq
	\A^{(1)}
	\supseteq
	\G^{(2)}
	\supseteq
	\C^{(2)}
	\supseteq
	\cdots
	\supseteq
	\C^{(i)}
	\supseteq
	\A^{(i)}
	\supseteq
	\G^{(i+1)}
	\supseteq
	\C^{(i+1)}
	\supseteq
	\cdots,
\]
all with identical upper growth rates. Although $\C$ needn't be well-quasi-ordered, $\G^{(1)}$ is, and so this chain must terminate by Proposition~\ref{prop-wqo-subclasses-dcc}. Thus at some point we must have $\A^{(i)}=\G^{(i+1)}$. This class, which has the same upper growth rate as $\C$, is both atomic and grid irreducible. Therefore, by Proposition~\ref{prop-atomic-grid-irreduce}, it is either sum or skew closed, and thus the upper growth rate of $\C$ is equal to the greatest upper growth rate of its sum or skew closed subclasses. As sum or skew closed classes have proper growth rates (Proposition~\ref{prop-arratia-gr}), it follows that $\C$ does as well, completing the proof.
\end{proof}

\section{Concluding Remarks}

After establishing several delicate results about sum indecomposable permutations in a subsequent paper with Pantone~\cite{pantone:growth-rates-of:}, we are able to provide a complete characterization of growth rates under $\xi$, leading to the number line shown in Figure~\ref{fig-set-of-growth-rates} of the introduction. Together with Bevan's Theorem~\ref{thm-lambda-B}, which shows that every real number greater than $\lambda_B\approx 2.36$ is the growth rate of a permutation class, this leaves us tantalizingly close to the complete characterization of growth rates of permutation classes---the gap between the two results is approximately $0.05176$.

In order to extend the approach used here and in \cite{pantone:growth-rates-of:} to complete this characterization, it seems that one would first want to establish the following.

\begin{conjecture}
\label{conj-gr-sum-closed}
Every upper growth rate of a permutation class is the growth rate of a sum closed class.
\end{conjecture}

Before moving on, a few notes regarding Conjecture~\ref{conj-gr-sum-closed} are in order. First, as we have observed in Corollary~\ref{cor-slicing-gr}, this conjecture can only be true because of coincidences. For example, the class $\Grid(1\ 1)$ has growth rate $2$, but this is not the growth rate of a sum closed subclass of this class. However, it happens that the class of \emph{layered permutations}, $\bigoplus\Av(12)$, also has growth rate $2$.

Second, Conjecture~\ref{conj-gr-sum-closed} is known, from \cite{bevan:intervals-of-pe:,vatter:permutation-cla:lambda:}, to hold for all growth rates between $\lambda_B\approx 2.35$ and approximately $3.79$ because all real numbers in this range are growth rates of sum closed classes (the growth rates above $3.79$ are achieved in \cite{vatter:permutation-cla:lambda:} via ``juxtapositions''). It seems highly unlikely that Conjecture~\ref{conj-gr-sum-closed} could fail above this range, though that has yet to be established, and has no bearing on the characterization of growth rates of permutation classes.

Were Conjecture~\ref{conj-gr-sum-closed} to be established, the classification of growth rates under $\lambda_B$ would still be far from a fait accompli---as discussed in \cite{pantone:growth-rates-of:}, the characterization of growth rates of even the limited collection of sum closed classes becomes increasingly difficult as the allowed growth rates increase (consider that the notorious class of $4231$-avoiding permutations is sum closed).

We save the discussion of those issues for \cite{pantone:growth-rates-of:} and for the rest of this conclusion discuss the challenges in establishing Conjecture~\ref{conj-gr-sum-closed}. The approach outlined in this paper can be roughly described as consisting of three steps: first, grid the class; second, refine this gridding by slicing in order to show that the growth rate of the class comes from a single cell (or is integral); third, repeat this process a finite number of times and use the structure of the cell classes to argue that the process terminates and leaves us with a sum (or skew) closed class.

While the first step of this approach can be extended as far as desired---Proposition~\ref{prop-G-gamma-grids} gives a process to determine $\G_\gamma$ for all growth rates $\gamma$---the second and third steps rely crucially on the properties of this cell class. We summarize these in the chart below.
\[
	\begin{tabular}{lll}
	\hline
	property of $\G_\gamma$&holds for $\gamma$&reference
	\\\hline
	finitely many simple permutations&$<\kappa\approx 2.21$& Proposition~\ref{prop-Ggamma-simples}\\
	$\subseteq\langle\O\rangle$&$\lessapprox 2.24$&Proposition~\ref{Gkappa-first}\\
	well-quasi-ordered&$<\xi\approx 2.31$&Theorem~\ref{thm-G-gamma-wqo}\\
	%finitely based\\
	$\subseteq\langle\Av(321)\cup\Av(123)\cup\{25314, 41352\}\rangle$&$\lessapprox 2.36$ ($>\lambda_B$)&Proposition~\ref{prop-Gxi-first}\\
	concentrated&$<1+\varphi\approx 2.62$ (?)&Conjecture~\ref{conj-concentrated-phi}\\
	bounded substitution depth&$<1+\varphi\approx 2.62$&Proposition~\ref{prop-bdd-alts}\\
	\hline
	\end{tabular}
\]
As this chart shows, the first obstacle to establishing Conjecture~\ref{conj-gr-sum-closed}, at least via the slicing technique used here (Theorem~\ref{thm-slicing}), would be establishing that $\G_{\lambda_B}$ is concentrated. As suggested before the proof of Theorem~\ref{thm-Gxi-concentration}, it may be possible to use a similar technique to establish this (note that the simple permutations of $\G_{\lambda_B}$ are concentrated by Proposition~\ref{prop-SiGxi-concentration}). Conjecture~\ref{conj-concentrated-phi} presents a loftier goal, which is possibly of independent interest.

If we could establish that $\G_{\lambda_B}$ is concentrated, we could then appeal to Corollary~\ref{cor-slicing-gr} to see that every upper growth rate below $\lambda_B$ is either integral (i.e., $0$, $1$, or $2$) or equal to the upper growth rate of a subclass of $\G_{\lambda_B}$. However, we know that $\G_{\xi}$, let alone $\G_{\lambda_B}$, contains infinite antichains, and thus we could not  apply the techniques of Section~\ref{sec-phase-transition} to show that these upper growth rates are equal to the growth rates of sum closed subclasses of $\G_{\lambda_B}$. 

For example, if we label the members of the antichain $U^o$ by length as $\{\mu_5,\mu_7,\dots\}$, then one class that illustrates the difficulties above $\xi$ is
\[
	\C=\Sub(\mu_7\oplus\mu_9\oplus\mu_{11}\oplus\cdots).
\]
This class is $\G$-griddable for all classes $\G$ containing $\bigoplus\PropSub(U^o)$ and omitting only finitely many members of $U^{o}$, where
\[
	\PropSub(U^o)
	=
	\{\pi\st\mbox{$\pi<\mu$ for some $\mu\in U^o$}\}
\]
denotes the \emph{proper downward closure} of $U^o$. However, no such class is grid irreducible, and thus the approach used in the proof of Theorem~\ref{thm-xi-main} does not terminate when applied to $\C$. It follows that the class $\C$ is not $\G$-griddable for a grid irreducible subclass, showing that the well-quasi-order hypothesis of Proposition~\ref{prop-grid-irreduce-wqo} cannot be removed.

Nevertheless, this is not a counterexample to Conjecture~\ref{conj-gr-sum-closed}, which we briefly demonstrate. First, $\C$ contains $\D=\bigoplus\PropSub(U^o)$, so its lower growth rate is at least $\xi$. To obtain an upper bound, observe that
\[
	\C
	\subseteq
	\D \oplus \Sub(\mu_7) \oplus \D \oplus \Sub(\mu_9) \oplus \D \oplus \cdots.
\]
Letting $f$ denote the generating function for the class on the right and $g$ denote the generating function for $\D=\bigoplus\PropSub(U^o)$, both counting the empty permutation, we see that
\[
	f=g(1+x^7g)(1+x^9g)\cdots.
\]
By elementary analysis, $f$ converges if and only if both $g$ and
\[
	x^7g+x^9g+\cdots=\frac{x^7g}{1-x^2}
\]
converge. It follows that the upper growth rate of $\C$ is at most $\xi$, so $\gr(\C)$ exists and is equal to the growth rate of a sum closed class because $\xi$ is the growth rate of $\D=\bigoplus\PropSub(U^o)$.

\begin{figure}
\begin{center}
	\begin{tikzpicture}[scale=0.2]
		\draw[lightgray, fill, rotate around={-45:(2,3)}] (2,3) ellipse (20pt and 60pt);
		\draw[lightgray, fill, rotate around={-45:(10.5,11.5)}] (10.5,11.5) ellipse (20pt and 40pt);
		\plotpermgraph{2,3,4,6,1,8,5,10,7,11,12,9};
		\plotpermbox{0.5}{0.5}{12.5}{12.5};
	\end{tikzpicture}
\quad\quad
	\begin{tikzpicture}[scale=0.2]
		\draw[lightgray, fill, rotate around={-45:(2,3)}] (2,3) ellipse (20pt and 60pt);
		\draw[lightgray, fill, rotate around={-45:(11,12)}] (11,12) ellipse (20pt and 60pt);
		\plotpermgraph{2,3,4,6,1,8,5,10,7,11,12,13,9};
		\plotpermbox{0.5}{0.5}{13.5}{13.5};
	\end{tikzpicture}
\quad\quad
	\begin{tikzpicture}[scale=0.2]
		\draw[lightgray, fill, rotate around={-45:(2.5,3.5)}] (2.5,3.5) ellipse (20pt and 80pt);
		\draw[lightgray, fill, rotate around={-45:(11.5,12.5)}] (11.5,12.5) ellipse (20pt and 40pt);
		\plotpermgraph{2,3,4,5,7,1,9,6,11,8,12,13,10};
		\plotpermbox{0.5}{0.5}{13.5}{13.5};
	\end{tikzpicture}
\quad\quad
	\begin{tikzpicture}[scale=0.2]
		\draw[lightgray, fill, rotate around={-45:(2.5,3.5)}] (2.5,3.5) ellipse (20pt and 80pt);
		\draw[lightgray, fill, rotate around={-45:(12,13)}] (12,13) ellipse (20pt and 60pt);
		\plotpermgraph{2,3,4,5,7,1,9,6,11,8,12,13,14,10};
		\plotpermbox{0.5}{0.5}{14.5}{14.5};
	\end{tikzpicture}
\end{center}
\caption{Infinite antichains in $\G_{\lambda_B}\setminus\G_\xi$. From left to right, we refer to these as types $(3,2)$, $(3,3)$, $(4,2)$, and $(4,3)$.}
\label{fig-G-lambda-antichains}
\end{figure}
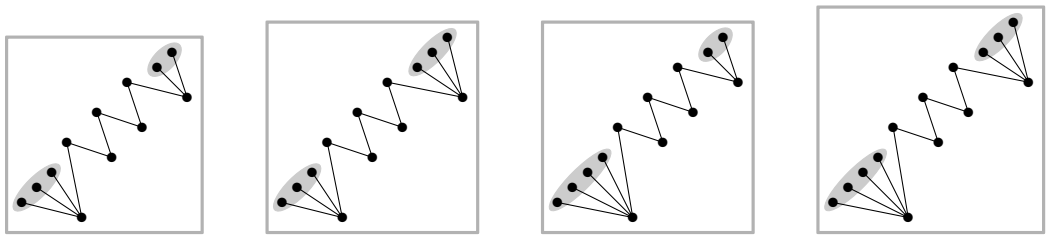

To summarize, a new technique must be developed to analyze subclasses of the cell class $\G_{\lambda_B}$. As an initial step in this direction, it should be possible to characterize all infinite antichains in $\G_{\lambda_B}$. In addition to $U$, four infinite antichains contained in this cell class are shown in Figure~\ref{fig-G-lambda-antichains} (all symmetries of these antichains are also contained in $\G_{\lambda_B}$). That these antichains are contained in $\G_{\lambda_B}$ is confirmed by the table below.
\[
	\begin{array}{lllll}
	\hline
	\mbox{antichain}&\mbox{sequence of sum indecomposables}&\mbox{lies in $\G_\gamma$ for}
	\\\hline
	\text{type $(2,2)$, i.e., $U^o$}&1,1,2,3,4^i,3,2,1&\gamma<\xi\approx 2.31 %2.30522
	\\
	\text{type $(3,2)$}&1,1,2,3,5^i,4,3,2,1&\gamma\lessapprox 2.33 %2.33081
	\\
	\text{type $(3,3)$}&1,1,2,3,5,6^i,5,4,3,2,1&\gamma\lessapprox 2.34 %2.34117
	\\
	\text{type $(4,2)$}&1,1,2,3,5,6^i,5,4,3,2,1&\gamma\lessapprox 2.34 %2.34117
	\\
	\text{type $(4,3)$}&1,1,2,3,5,7^i,6,5,4,3,2,1&\gamma\lessapprox 2.35 %2.35114
	\\\hline
	\end{array}
\]

In fact, the antichains depicted in Figure~\ref{fig-G-lambda-antichains} are precisely those used by Bevan to establish Theorem~\ref{thm-lambda-B}. Note that all of the anchors of these antichains (the entries enclosed by ovals) are increasing. One might wonder if infinite antichains based on the increasing oscillating sequence but with decreasing anchors, as depicted in Figure~\ref{fig-G-lambda-antichains-NOT}, lie in cell classes below $\lambda_B$. The answer is no; the sequence of sum indecomposables contained a member of either antichain from Figure~\ref{fig-G-lambda-antichains-NOT} is of the form $1,1,3,4^i,3,2,1$. Therefore the sum closure of the proper downward closure of either antichain has the generating function
\[
	\frac{1}{1-\left(x+x^2+3x^3+\frac{4x^4}{1-x}\right)}
	=
	\frac{1-x}{1-2x-2x^3-x^4},
\]
which has a growth rate of precisely $1+\sqrt{2}$. Therefore the cell class $\G_{\lambda_B}$ has only finite intersection with such antichains.

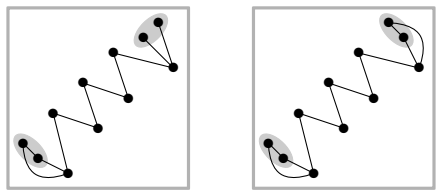
\begin{figure}
\begin{center}
	\begin{tikzpicture}[scale=0.2]
		% The ovals around the intervals:
		\draw[lightgray, fill, rotate around={45:(1.5,2.5)}] (1.5,2.5) ellipse (20pt and 40pt);
		\draw[lightgray, fill, rotate around={-45:(9.5,10.5)}] (9.5,10.5) ellipse (20pt and 40pt);
		% The permutation:
		\plotperm{3,2,5,1,7,4,9,6,10,11,8};
		\plotpermbox{0.5}{0.5}{11.5}{11.5};
		% Because we want one curved edge, we have to draw all the straight edges by hand.
		\draw (1,3)--(2,2)--(4,1)--(3,5)--(6,4)--(5,7)--(8,6)--(7,9)--(11,8);
		\draw (9,10)--(11,8)--(10,11);
		\draw plot [smooth, tension=1] coordinates { (1,3) ({2.5-0.35*2}, {2-0.35*3}) (4,1) };
	\end{tikzpicture}
\quad\quad
	\begin{tikzpicture}[scale=0.2]
		% The ovals around the intervals:
		\draw[lightgray, fill, rotate around={45:(1.5,2.5)}] (1.5,2.5) ellipse (20pt and 40pt);
		\draw[lightgray, fill, rotate around={45:(9.5,10.5)}] (9.5,10.5) ellipse (20pt and 40pt);
		% The permutation:
		\plotperm{3,2,5,1,7,4,9,6,11,10,8};
		\plotpermbox{0.5}{0.5}{11.5}{11.5};
		% Because we want two curved edges, we have to draw all the straight edges by hand.
		\draw (1,3)--(2,2)--(4,1)--(3,5)--(6,4)--(5,7)--(8,6)--(7,9)--(11,8)--(10,10)--(9,11);
		\draw plot [smooth, tension=1] coordinates { (1,3) ({2.5-0.35*2}, {2-0.35*3}) (4,1) };
		\draw plot [smooth, tension=1] coordinates { (9,11) ({10+0.35*3}, {9.5+0.35*2}) (11,8) };
	\end{tikzpicture}
\end{center}
\caption{Two infinite antichains which have only finite intersection with $\G_{1+\sqrt{2}}$.}
\label{fig-G-lambda-antichains-NOT}
\end{figure}

Even if the complete characterization of growth rates of permutation classes remains out of reach, it may also be possible by following this approach to establish that $\lambda_B$ is best possible, i.e., that in every open neighborhood of $\lambda_B$ there is a real number that is \emph{not} the growth rate of a permutation class.

\bigskip

\minisec{Acknowledgements}
I am grateful to Michael Albert, David Bevan, Robert Brignall, Michael Engen, Jay Pantone, and the anonymous referees for their many insightful comments on this work.

%\bibliographystyle{acm}
%\bibliography{xi}

\end{document}